\documentclass[12pt,reqno,a4paper]{amsart}
\usepackage{amssymb,latexsym}
\usepackage{amsthm, amsmath, amsfonts}
\usepackage{hyperref}
\usepackage{empheq}
\usepackage{chngcntr}
\usepackage{apptools}
\usepackage{graphicx}  
\usepackage{dcolumn}   
\usepackage{bm}        
\usepackage{enumerate}
\usepackage{slashed}
\usepackage{simplewick}
\usepackage{comment}
\usepackage{color}
\usepackage{soul}
\usepackage[english]{babel}
\usepackage{appendix}
\usepackage[utf8]{inputenc}
\usepackage{tikz} 
\usetikzlibrary{decorations.markings}
\usepackage{enumitem}
\numberwithin{equation}{section}

\usepackage{marginnote}
\usepackage[normalem]{ulem}

\def\bz{\bar z}
\def\bi{\bar i}
\def\bj{\bar j}
\def\bk{\bar k}
\def\bl{\bar l}
\def\br{\bar r}

\def\p{\partial}
\def\bp{\bar\partial}
\def\a{\alpha}
\def\b{\beta}

\def\bcal_k{\mathcal B_k}

\newcommand{\kahler}{K\"ahler }

\def\tr{{\rm tr}}
\DeclareMathOperator{\Td}{Td}
\DeclareMathOperator{\ch}{ch}
\DeclareMathOperator{\cc}{c}
\DeclareMathOperator{\Tr}{Tr}
\DeclareMathOperator{\ric}{Ric}
\DeclareMathOperator{\End}{End}

\def\({\left(}
\def\){\right)}

\def\Im{{\rm Im\,}}
\def\re{{\rm Re\,}}

\newtheorem{thm}{{\sc Theorem}}[section]

\newtheorem*{maindefn*}{{\sc Definition}}
\newtheorem{theo}{{\sc Theorem}}[section]

\newtheorem{maintheo}{{\sc Theorem}}

\newtheorem{mainprop}{{\sc Proposition}}

\newtheorem{lem}[theo]{{\sc Lemma}}
\newtheorem{prop}[theo]{{\sc Proposition}}

\theoremstyle{remark}
\newtheorem{remark}{Remark}

\newenvironment{Ack}{\medskip\noindent{\it Acknowledgements:\/} }{\medskip}

\newtheoremstyle{dotless}{}{}{\itshape}{}{\bfseries}{}{ }{}
\theoremstyle{dotless}

\newtheorem{definition}[theo]{{\sc Definition}}

\title[Partition functions of DPP on polarized K\"ahler manifolds]{Partition functions of determinantal point processes on polarized K\"ahler manifolds}
\author{Kiyoon Eum}
\address{Department of Mathematical Sciences, KAIST, 291 Daehak-ro, Yuseong-gu, Daejeon 34141, South Korea}
\email{kyeum@kaist.ac.kr}

\begin{document}

\begin{abstract}
In this paper, we study the full asymptotic expansion of the partition functions of determinantal point processes defined on a polarized \kahler manifold. We show that the coefficients of the expansion are given by geometric functionals on \kahler metrics satisfying the cocycle identity, whose first variations can be expressed through the TYZ expansion coefficients of the Bergman kernel. In particular, these functionals naturally generalize the Mabuchi functional in \kahler geometry and the Liouville functional on Riemann surfaces. We further show that Futaki-type holomorphic invariants obstruct the existence of critical points of these geometric functionals, extending Lu's formula. We also verify that certain formulas remain valid up to the third coefficient without assuming polarization. Finally, we discuss the relation of our results to the quantum Hall effect (QHE), where the determinantal point process provides a microscopic model. In particular, we recover the higher-dimensional effective Chern-Simons actions derived in the physics literature and confirm a conjecture of Klevtsov on the form of the partition function asymptotics.
\end{abstract}

\maketitle

\thispagestyle{empty}
\tableofcontents


\section{Introduction}

In this paper, we study the asymptotic expansion of the partition functions of determinantal point processes defined on a polarized K\"ahler manifold. These processes were introduced by Berman in a series of papers \cite{berman2013determinantal,berman2014determinantal,berman2017large} and have been shown to satisfy a large deviation principle.

We derive a full asymptotic expansion of the partition function in two ways: one using Bergman kernel asymptotics and the other using the Quillen anomaly formula along with the asymptotic expansion of the Ray-Singer analytic torsion. By combining these two expressions, we analyze the properties of each coefficient in the asymptotic expansion. Each coefficient is given by geometric functionals on K\"ahler metrics that satisfy the cocycle identity, and its first variation can be expressed in terms of the coefficients of the Bergman kernel asymptotics. In particular, the second coefficient corresponds to the Mabuchi functional in K\"ahler geometry, while the third coefficient generalizes the Liouville functional on Riemann surfaces to higher dimensions.

Furthermore, we show that the holomorphic invariant introduced by Futaki \cite{futaki2004asymptotic} obstructs the existence of critical points for each geometric functional given by the coefficients of the asymptotic expansion. Although the construction relies on the polarization of the K\"ahler class, we demonstrate through direct computation that the results remain valid up to the third coefficient, without the polarization assumption. We will also discuss the connection between our results and the physics literature.

\subsection{Mathematical background}

Let $(M,\omega)$ be a polarized K\"ahler manifold. That is, there is a Hermitian line bundle $(L,h)$ over $M$ such that its Chern curvature form is given by $-i\omega$. Let $d_k:=\dim H^0(M,L^k)$. The determinantal point process (DPP) associated with $L^k$ is a probability measure $\mu_{k,\omega}$ on $M^{d_k}$ defined as
\begin{equation*}
\mu_{k,\omega}:=\frac{1}{Z_k[h]}|\Psi^k(z_1,\cdots,z_{d_k})|^2_{(h^k)^{\boxtimes d_k}}\prod_{i=1}^{d_k}\frac{\omega^n}{n!},
\end{equation*}
where $\Psi^k$ is a holomorphic section of $(L^k)^{\boxtimes d_k}$ over $M^{d_k}$ given by the Slater determinant. A more detailed definition will be given in Section \ref{sectiondpp}. It was introduced by Berman and subsequently studied by him, particularly regarding its relation to canonical Kähler metrics \cite{berman2018kahler,berman2022measure,berman2024probabilistic}. See also \cite{fujita2016berman,fujita2018k,aoi2024microscopic} for related results. 

The normalization factor $Z_k[h]$, also known as the partition function, is essentially Donaldson's $\mathcal{L}_k$-functional introduced in \cite{donaldson2005scalar}. The first term in the large $k$ asymptotics of $Z_k[h]$ without assuming positivity of the curvature of $h$, was studied by Berman-Boucksom \cite{berman2010growth}, extending the method of \cite{donaldson2005scalar}. For a recent extension of their results, see \cite{finski2025small}. In this paper, we focus on the (smooth) positive curvature case and derive the full asymptotic expansion of $Z_k[h]$ as $k\rightarrow\infty$, identifying its coefficients with geometric functionals on K\"ahler metrics.

The relation between the asymptotic expansion of $Z_k[h]$ and Bergman kernel asymptotics was given in \cite{donaldson2005scalar}. Following \cite{klevtsov2017quantum,shen2025geometric}, we have an alternative way to obtain an asymptotic expansion of $Z_k[h]$ using the Quillen anomaly formula \cite{bismut1988analytic} and the asymptotic expansion of the Ray-Singer analytic torsion \cite{vasserot1989asymptotics,finski2018full}. We will review the Quillen anomaly formula and Ray-Singer analytic torsion in Section \ref{subsec}. Combining these two approaches, we can study the properties of the coefficients in the asymptotic expansion. 

From the Quillen anomaly formula approach, the coefficients are expressed as geometric functionals involving Bott-Chern forms. Bott-Chern forms were first used by Donaldson \cite{donaldson1985anti,donaldson1987infinite} to construct geometric functionals and were further developed by Tian \cite{tian1994k,tian1999bott} in the context of K\"ahler geometry. See also \cite{weinkove2002higher}. Bott-Chern forms and related materials will be reviewed in Section \ref{subsecBC}.

From the Bergman kernel asymptotics approach, the first variation of the geometric functionals is expressed in terms of integrals of scalar quantities involving the curvature of the K\"ahler metrics, arising from the Bergman kernel asymptotic expansion. In particular, the critical point equation is of the form $a_j-\Delta a_{j-1}\equiv \text{constant}$, where $a_j$ is the $j^{th}$ coefficient in the TYZ expansion of the Bergman kernel; see (\ref{Basymp}). Since the works of Catlin, Ruan, Tian, and Zelditch \cite{catlin1999bergman,ruan1996canonical,tian1990set,zelditch1998szego}, Bergman kernel asymptotics has become a well-studied subject with various approaches; see, for example, \cite{dai2006asymptotic,berman2008direct,douglas2010bergman,xu2012closed,hezari2016asymptotic}. Especially, we will use the explicit expression for $a_j, \;j\leq 3$ that Lu computed in \cite{lu2000lower}. By the result of \cite{lu2004log}, we know that the equation $a_j-\Delta a_{j-1}\equiv \text{constant}$ is an elliptic equation of order $2j+2$ in K\"ahler potential. 

In \cite{futaki2004asymptotic}, Futaki introduced a family of holomorphic invariants that generalize various integral invariants, including the Futaki invariant \cite{futaki1983compact} and the Bando-Futaki invariant \cite{bando2006obstruction}. We show that these invariants serve as obstructions to the existence of critical points of our geometric functionals $S_j$. As a byproduct, we extend Lu's formula $(4.4)$ in \cite{lu2004k} to all $j\geq 0$, where it was originally verified by direct computation up to $j=2$.

In the standard GIT picture of K\"ahler geometry, geometric functionals correspond to the log-norm functionals, while Futaki-type invariants correspond to the weight. See also \cite{foth2007manifold,eum2025asymptotic} for a moment map interpretation of the functions $a_j-\Delta a_{j-1}$. Thus, equations of the form $a_j-\Delta a_{j-1}\equiv \text{constant}$ formally fit into the GIT setting except the convexity of the log-norm functional. We compute the second variation of $S_2$ with a view to this problem. See \cite{cristofori2025third} and references therein for related work on equations of the form $a_j\equiv \text{constant}$.

Throughout this paper, we emphasize when the polarization assumption is not used. In particular, in the Appendix \ref{appendix}, we directly verify some of our results without the polarization assumption, following the spirit of original work of Mabuchi \cite{mabuchi1986k}.

\subsection{Relation to physics literature}

In the physics literature, the determinantal point process on Riemann surfaces corresponds to the plasma analogy in the quantum Hall effect (QHE). For a general physics background on the QHE, see \cite{tong2016lectures}; for the role of geometric functionals in the QHE, see \cite{can2015geometry,klevtsov2016geometry}. The form of asymptotic expansion of the corresponding partition function was conjectured by Zabrodin-Wiegmann in \cite{zabrodin2006large} for $\mathbb{CP}^1$ and later proved in \cite{klevtsov2014random,klevtsov2017quantum,shen2025geometric} for arbitrary Riemann surfaces with pure bulk assumption. For mathematical results in the presence of an edge, see \cite{serfaty2024lectures,byun2023partition,byun2025free} and references therein.

Recently, higher-dimensional QHE has also been studied in cases where the even-dimensional spatial manifold has a complex structure \cite{karabali2016geometry,karabali2023transport,agarwal2025fractional}. In particular, an effective action for higher-dimensional QHE was derived in terms of Chern-Simons forms. We provide an alternative derivation of this result in Section \ref{sectioncs}. In dimension $n=1$, Klevtsov-Ma-Marinescu-Wiegmann \cite{klevtsov2017quantum} derived the effective action in terms of the Chern-Simons functional considering a more general moduli space.

In \cite{klevtsov2014random}, based on direct computations in dimension $n=1$, Klevtsov conjectured the form of the asymptotic expansion of partition function in all dimensions. We confirm his conjecture in Section \ref{sectiondpp} (see Remark \ref{klev}). We also briefly discuss the relationship between the asymptotics of the QHE partition function and the 2D quantum gravity model in Section \ref{sectionpathint}. Since this discussion is independent of the other parts of the paper, we defer it to Section \ref{sectionpathint}.

\subsection{Statement of the main results}

Now we state our main results, starting with the asymptotic expansion of the DPP partition function. Note that in the main text, we will use a slightly different normalization convention, which will introduce additional $2\pi$ factors. Also we will identify a hermitian metric $h$ with K\"ahler potential $\varphi$. Let $\mathcal{K}_0$ be the space of K\"ahler forms in $[\omega]$ and $\mathcal{K}_\omega$ be the space of K\"ahler potentials.

\begin{maintheo}[Asymptotics of the Partition Functions]\label{m1}
$\log \frac{Z_k[\varphi]}{Z_k[0]}$ admits the following form of asymptotic expansion as $k\rightarrow\infty$:
\begin{equation}\label{ver2}
\log \frac{Z_k[\varphi]}{Z_k[0]}=kd_{k}S_0[\varphi,0]+k^{n}S_1[\omega_\varphi,\omega]+k^{n-1}S_2[\omega_\varphi,\omega]+k^{n-2}S_3[\omega_\varphi,\omega]+\cdots.
\end{equation}
For $j>0$, $S_j[\cdot,\cdot]:\mathcal{K}_0\times\mathcal{K}_0\rightarrow\mathbb{R}$ satisfies
\begin{enumerate}
    \item Cocyle identity : for any three K\"ahler metrics $\omega_2, \omega_1, \omega_0$ in $\mathcal{K}_0$,
    \begin{align*}
S_j[\omega_1,\omega_0]&=-S_j[\omega_0,\omega_1],\\   
S_j[\omega_2,\omega_0]&=S_j[\omega_2,\omega_1]+S_j[\omega_1,\omega_0];
    \end{align*}
    \item The derivative of $S_j[\omega_\varphi,\omega]$ on $\mathcal{K}_0=\mathcal{K}_\omega/\mathbb{R}$ is given by
\begin{equation}\label{firstvar}
\delta_\varphi S_j[\omega_\varphi,\omega]=\int_M \delta\varphi\left( \widehat{a_j(\omega_\varphi)}+\Delta_\varphi a_{j-1}(\omega_\varphi)-a_{j}(\omega_\varphi) \right) \frac{\omega_\varphi^n}{n!},
\end{equation}
where $\widehat{a_j(\omega_\varphi)}$ denotes the average of $a_j(\omega_\varphi)$
\begin{equation*}
\widehat{a_j(\omega_\varphi)}:=\frac{1}{V}\int_M a_j(\omega_\varphi)\frac{\omega_\varphi^n}{n!}=\frac{1}{V}\int_M \Td_j(T^{1,0}M)\ch_{n-j}(L),
\end{equation*}
which does not depend on $\omega_\varphi$.
    \item For $j>n+1$, $S_j[\omega_2,\omega_1]$ is an exact cocyle. That is, it can be written as a difference
\begin{equation*}
S_j[\omega_2,\omega_1]=s_j[\omega_2]-s_j[\omega_1],
\end{equation*}
for a local functional $s_j[\cdot] : \mathcal{K}_0\rightarrow\mathbb{R}$ of the metric.
\end{enumerate}
\end{maintheo}

One can check that $S_0$ is (minus of) the Aubin-Yau functional $I$, and $S_1$ is the Mabuchi functional $\mathcal{M}$ in K\"ahler geometry. We will show that $S_2$ is explicitly given by
\begin{align*}
S_2[\omega_\varphi,\omega]=&-i\int_M BC(\Td_2;\omega_\varphi,\omega)\frac{\omega^{n-1}}{(n-1)!}+\Td_{2}(R_\varphi)\frac{-i}{(n-1)!}\sum^{n-2}_{s=0}\varphi \omega_{\varphi}^s\wedge \omega^{n-2-s}\\
&+\frac{i}{V}\left(\int_M \Td_2(T^{1,0}M)\frac{\omega^{n-2}}{(n-2)!}\right)\times\int_M \frac{-i}{(n+1)!}\sum^{n}_{s=0}\varphi \omega_{\varphi}^s\wedge \omega^{n-s}.
\end{align*}
As we will see, $S_2$ reduces to the Liouville action in dimension $n=1$. We compute its second variation as follows:

\begin{mainprop}\label{m2}
Let $\varphi_t$ be a smooth path in $\mathcal{K}_\omega$ with $\varphi_0=0$. Then we have the following.
\begin{align*}
& \left.\frac{d^2}{dt^2}\right|_{t=0}S_2[\omega_t,\omega]=\int_M \ddot{\varphi}\left(\widehat{a_2}+\frac{1}{6}\Delta S-\frac{1}{24}\left(|R|^2-4|\ric|^2+3S^2\right) \right) \frac{\omega^n}{n!} +\widehat{a_2}\dot{\varphi}\Delta\dot{\varphi}\frac{\omega^n}{n!}  \\
&\;+ \frac{i}{6}\p\Delta\dot{\varphi}\wedge\bp\Delta\dot{\varphi}\wedge\frac{\omega^{n-1}}{(n-1)!}-\frac{i}{6}\p\dot{\varphi}\wedge\bp\dot{\varphi}\wedge i\p\bp S\wedge \frac{\omega^{n-2}}{(n-2)!}\\
&\;+\frac{1}{4}(\Delta\dot{\varphi})^2S\frac{\omega^{n}}{n!}+\frac{i}{8}\p\dot{\varphi}\wedge\bp\dot{\varphi}\wedge\ric^2\wedge\frac{\omega^{n-3}}{(n-3)!}+\frac{i}{24}\p\dot{\varphi}\wedge\bp\dot{\varphi}\wedge\Tr(R^2)\wedge\frac{\omega^{n-3}}{(n-3)!} \\
&\;-\frac{1}{2}\Delta\dot{\varphi}\langle i\p\bp\dot{\varphi},\ric \rangle \frac{\omega^{n}}{n!} +\frac{1}{12}\left( g^{i\bar{q}}g^{p\bar{j}}g^{r\bar{l}}g^{k\bar{s}} R_{p\bar{q}r\bar{s}} \p_i\p_{\bar{j}}\dot{\varphi} \p_k\p_{\bar{l}}\dot{\varphi} \right)\frac{\omega^{n}}{n!}. 
\end{align*}
\end{mainprop}

The following Theorem generalizes formula \cite[(4.4)]{lu2004k} for all $j\geq 0$. Let $X$ be a holomorphic vector field on $M$ and $\theta_X$ be a holomorphy potential function satisfying $\iota_X \omega = -\bp \theta_X$.

\begin{maintheo}\label{m3}
Let $(M,\omega)$ be a polarized K\"ahler manifold. Let $X$ and $\theta_X$ be given as above and suppose it is purely imaginary. Then for all $j\geq 0$, we have the following identity:
\begin{equation}\label{m3formula}
\int_M \theta_X \left( a_j(\omega)-\Delta a_{j-1}(\omega) \right) \frac{\omega^n}{n!}= \frac{1}{(n+1-j)!}\int_M \Td_j(R+\nabla X)(\omega+\theta_X)^{n+1-j}.
\end{equation}
More generally, for non purely imaginary $\theta_X$, we have
\begin{equation}\label{m3formula2}
\int_M i\Im \theta \left( a_j(\omega)-\Delta a_{j-1}(\omega) \right) \frac{\omega^n}{n!}= \frac{1}{(n+1-j)!}\int_M \Td_j(R+\frac12(\nabla X +(\overline{\nabla X})^{\flat,\sharp}))(\omega+i\Im \theta)^{n+1-j}.
\end{equation}
Under the normalization of $\theta_X$ by $\int_M \theta_X \omega^n =0$, the right hand sides of (\ref{m3formula}) and (\ref{m3formula2}) are independent of the choice of the K\"ahler metric in $\cc_1(L)$.
\end{maintheo}

The right hand side of (\ref{m3formula}) is precisely the holomorphic invariant introduced in \cite{futaki2004asymptotic}. We will show that these invariants obstruct the existence of critical points of $S_j$. When there is no holomorphic vector fields, these obstructions vanish. In this context, assuming that $Aut(M,L)$ is discrete, we can prove the following proposition using Donaldson's result \cite{donaldson2001scalar} on balanced metrics.

\begin{mainprop}\label{m4}
Suppose that $Aut(M,L)$ (modulo trivial action of $\mathbb{C}^*$) is discrete. If $\omega_\infty$ is a critical point of $S_1$ in $\mathcal{K}_0$, then it is a critical point of $S_j$ for all $j>0$.
\end{mainprop}

Finally, we derive a formula for the effective action of the higher-dimensional QHE, as presented in \cite{karabali2016geometry}, expressed in terms of Chern-Simons forms.

\begin{mainprop}\label{m5}
The effective action for the higher-dimensional quantum Hall effect associated with $L$ is given by
\begin{equation}\label{eff}
S_{eff}=2\int_{M\times [0,1]} \left[\Td(R_{\mathbf{T'M}}(\boldsymbol{\omega}))CS(\ch;\boldsymbol{\nabla}_\mathbf{L},\boldsymbol{\nabla^0}_\mathbf{L})+ CS(\Td;\boldsymbol{\nabla}_\mathbf{T'M},\boldsymbol{\nabla^0}_\mathbf{T'M})\ch(R_\mathbf{L}(\mathbf{h_0})) \right]_{2n+1}+\widetilde{S}.
\end{equation}
As we replace $L$ with $L^k$ and send $k\rightarrow \infty$, the leading order ($k^{n+1}$) term of the effective action is given by 
\begin{equation*}
2\int_{M\times[0,1]} CS(\ch_{n+1};\boldsymbol{\nabla}_\mathbf{L},\boldsymbol{\nabla^0}_\mathbf{L}).
\end{equation*}
\end{mainprop}

\subsection{Structure of the paper}

\begin{itemize}
  \item Section \ref{sectionprelim} contains preliminaries on Bergman kernel asymptotics, Chern-Simons and Bott-Chern forms, Quillen metrics, and Ray-Singer analytic torsion. 
  \item Section \ref{sectiondpp} defines the determinantal point process, and contains the proof of Theorem \ref{m1}.
  \item Section \ref{sectionFutaki} contains the proof of Theorem \ref{m3} and Proposition \ref{m4}.
  \item Section \ref{sectionpathint} briefly digresses into the relationship between the asymptotics of the partition function and the 2D quantum gravity model.
  \item Section \ref{sectioncs} contains the proof of Proposition \ref{m5}.
  \item Appendix \ref{appendix} provides some explicit computations on $S_2$ and contains the proof of Proposition \ref{m2}.
\end{itemize}

\begin{Ack}
The author would like to thank Siarhei Finski for clarifications regarding his Theorem and Remark \ref{Finskirmk}.
This work was supported by the National Research Foundation of Korea (NRF) grant funded by the Korea government(MSIT) RS-2024-00346651.
\end{Ack}

\section{Preliminaries}\label{sectionprelim}

\subsection{Bergman kernel asymptotics}
In this paper, we consider $n$-dimensional polarized compact K\"ahler manifold $(M,g,\omega)$, $g$ is a K\"ahler metric, $\omega$ is a K\"ahler form with Hermitian ample line bundle $(L,h)$ whose Chern curvature form is given by $R(h)=-i\omega$. For each $k\in\mathbb{N}$, $h$ induces a Hermitian metric $h^k$ on $L^k$.
Together with K\"ahler volume form $\omega^n/n!$, it induces the $L^2$ metric on $H^0(M,L^k)$. The diagonal of the $k^{th}$ Bergman kernel $\rho_k(\omega)$ associated with these data is defined by
\begin{equation*}
\rho_k(\omega)=\sum_{i=1}^{d_k} ||S_i^k||_{h^k}^2 \in C^\infty(M,\mathbb{R}),
\end{equation*}
where $d_k=\dim H^0(M,L^k)$ and $\left(S_i^k\right)_{i=1}^{d_k}$ is any orthonormal basis of $H^0(M,L^k)$ with respect to $L^2$ metric induced by $h^k$ and $\omega^n/n!$. 

Since the initial works of Catlin, Ruan, Tian, and Zelditch, it is now well known that the diagonal of the Bergman kernel admits a full asymptotic expansion (called TYZ expansion) in $k$:
\begin{equation}\label{Basymp}
(2\pi)^n\rho_k(\omega)(x)=a_0(\omega)(x)k^n+a_1(\omega)(x)k^{n-1}+a_2(\omega)(x)k^{n-2}+\cdots,
\end{equation}
where $a_j(\omega)$ are smooth coefficients given by universal polynomials in curvature $R$ of $g$ and its derivatives with order $\leq 2j-2$ at $x$. For notational convenience, set $a_{-1}:=0$. Note that $a_j(\omega)$ can be understood as a formal expression even when $[\omega]$ is not polarized. For the precise meaning of the asymptotic expansion (\ref{Basymp}) and a comprehensive account of the subject, see \cite{ma2007holomorphic}.

Here we adopt the convention in which there are no $2\pi$ factors in $a_j$s. In this convention, integrating $(2\pi)^n\rho_k$ over $M$, we get the following form of the asymptotic Riemann-Roch-Hirzebruch formula:
\begin{equation}\label{aRR}
(2\pi)^n\dim H^0(M,L^k)=(2\pi)^n\int_M \rho_k(\omega)\frac{\omega^n}{n!}  = k^n\int_M a_0(\omega)\frac{\omega^n}{n!}+k^{n-1}\int_M a_1(\omega) \frac{\omega^n}{n!} + \cdots.
\end{equation}
This motivates the slightly unconventional modification of $\Td$ and Chern character forms which we will introduce in the next subsection. The only purpose of it is to prevent the $2\pi$ factors from clogging the formulas.

In \cite{lu2000lower}, Lu explicitly computed the first four coefficients $a_1, a_2, a_3$ of the expansion as ($a_0=1$ was already noted in \cite{zelditch1998szego}):
\begin{equation}\label{lower}
\left\{ \begin{aligned} 
  a_0 &= 1\\
  a_1 &= \frac12 S\\
  a_2 &= \frac13 \Delta S + \frac{1}{24}(|R|^2-4|\ric|^2+3S^2)\\
  a_3 &= \frac18\Delta\Delta S + \cdots\\
\end{aligned} \right.
\end{equation}
where $R, \ric, S$ are the Riemann curvature, Ricci curvature, and scalar curvature of the K\"ahler metric $g$, respectively, and $\Delta$ is one-half of the Riemannian Laplacian, defined by $\Delta f=g^{k\bl} \p_k \p_{\bl} f$. In particular, $a_1=\frac12 S$ played an important role in \cite{donaldson2001scalar}. The explicit formula for $a_2$ will be used in this paper.

\subsection{Invariant polynomials and secondary characteristic forms}\label{subsecBC}
Let $M'$ be a $n'$ dimensional compact smooth manifold and $E$ be a $\mathbb{C}^r$ vector bundle over $M'$. Given a connection $\nabla$ on $E$, denote its curvature form by $R(\nabla)$. For any symmetric $GL(r,\mathbb{C})$-invariant $p$-linear function $\phi$ on $\mathfrak{gl}(r,\mathbb{C})$, we get a Chern-Weil form
\begin{equation*}
\phi(R(\nabla)):=\phi(R(\nabla),R(\nabla),\cdots,R(\nabla))\in \Omega^{2p}(M'),
\end{equation*}
representing a characteristic class in $H^{2p}(M',\mathbb{C})$. Usually $\phi(A):=\phi(A,\cdots,A)$ is called an invariant polynomial in $A\in\mathfrak{gl}(r,\mathbb{C})$ and is identified with its polarization. For invariant polynomial $\phi$ and any connection $\nabla$ on $E$, we denote $\phi(E):=\phi(R(\nabla))$. Invariant polynomials important to us are the Chern polynomial $\cc$, the Chern character polynomial $\ch$, and the Todd polynomial $\Td$, which are defined as follows. For $A\in\mathfrak{gl}(r,\mathbb{C})$, define
\begin{align*}
\cc(A)&:=\det (I+iA)=1+\cc_1(A)+\cc_2(A)+\cdots;\\
\ch(A)&:=\Tr (e^{iA}) = r+\ch_1(A)+\ch_2(A)\cdots;\\
\Td(A)&:=\det\frac{iA}{1-e^{-iA}}=1+\Td_1(A)+\Td_2(A)+\cdots.
\end{align*}
The corresponding $p$-linear functions are then defined by polarization. Note that we omit $2\pi$ factors from their standard definitions altogether so that the asymptotic Riemann-Roch-Hirzebruch formula for $L^k$ as $k\rightarrow\infty$ now reads (which follows from Riemann-Roch-Hirzebruch theorem and Kodaira-Serre vanishing theorem):
\begin{equation}\label{aRR2}
(2\pi)^n\dim H^0(M,L^k)=k^n\int_M \Td_0(T^{1,0}M)\ch_n(L)+k^{n-1}\int_M \Td_1(T^{1,0}M)\ch_{n-1}(L)+\cdots.
\end{equation}
Comparing it to (\ref{aRR}), we get 
\begin{equation*}
\int_M a_j(\omega)\frac{\omega^n}{n!}=\int_M \Td_j(T^{1,0}M)\ch_{n-j}(L)
\end{equation*}
for $j=0,1,\cdots$.

Let $I_p(r)$ be the space of all symmetric $GL(r,\mathbb{C})$-invariant $p$-linear functions. Put $I(r):=\oplus_{p\geq 0}I_p(r)$. Denote the space of connections on $E$ by $\mathcal{A}_E$. Then we have the following constructions, called secondary characteristic forms (see \cite{BC} and \cite{pingali2014bott}).
\begin{prop}\label{cs}
There is a well-defined map
\begin{equation*}
CS : I(r)\times \mathcal{A}_E \times \mathcal{A}_E \rightarrow \oplus_{p\geq0}\Omega^{2p}(M')/\Im d
\end{equation*}
defined as follows: For any $\phi\in I_p(r)$,
\begin{equation*}
CS(\phi;\nabla^1,\nabla^0)=\int_0^1 \phi'(R(\nabla^t);\dot{\nabla^t})\,dt,
\end{equation*}
where $\nabla^t$ is any smooth path in $\mathcal{A}_E$ joining $\nabla^0$ to $\nabla^1$, and $\phi'(A;B)$ is a shorthand notation for $\sum_\a \phi(A,...,B_\a,...,A)$. Such a $CS(\phi;\nabla^1,\nabla^0)$ is called the Chern-Simons form. Chern-Simons forms satisfy
\begin{enumerate}
    \item $CS(\phi;\nabla,\nabla)=0$ and for any three connections $\nabla^2, \nabla^1, \nabla^0$ in $\mathcal{A}_E$, 
    \begin{equation*}
CS(\phi;\nabla^2,\nabla^0)=CS(\phi;\nabla^2,\nabla^1)+CS(\phi;\nabla^1,\nabla^0) ;
    \end{equation*}
    \item $dCS(\phi;\nabla^1,\nabla^0)=\phi(R(\nabla^1))-\phi(R(\nabla^0)) ;$
    \item If $\nabla^t$ is any smooth path in $\mathcal{A}_E$, we have
    \begin{equation*}
\frac{d}{dt}CS(\phi;\nabla^t,\nabla)=\phi'(R(\nabla^t);\dot{\nabla^t}).
    \end{equation*}
\end{enumerate}
\end{prop}

Now assume $M'$ is a complex manifold and $E$ is a holomorphic vector bundle. For each Hermitian metric $h$ on $E$, let $R(h)$ be a Chern curvature form associated to $h$. Denote the space of Hermitian metrics on $E$ by $\mathcal{H}_E$. Then we have the following.

\begin{prop}\label{bc}
There is a well-defined map
\begin{equation*}
BC : I(r)\times \mathcal{H}_E \times \mathcal{H}_E \rightarrow \oplus_{p\geq0}\Omega^{p,p}(M')/\Im \p + \Im \bp
\end{equation*}
defined as follows: For any $\phi\in I_p(r)$,
\begin{equation}\label{homotop}
BC(\phi;h_1,h_0)=\int_0^1 \phi'(R(h_t);\dot{h_t}h_t^{-1})\,dt,
\end{equation}
where $h_t$ is any smooth path in $\mathcal{H}_E$ joining $h_0$ to $h_1$. Such a $BC(\phi;h_1,h_0)$ is called the Bott-Chern form. Bott-Chern forms satisfy
\begin{enumerate}
    \item $BC(\phi;h,h)=0$ and for any three metrics $h_2, h_1, h_0$ in $\mathcal{H}_E$, 
    \begin{equation*}
BC(\phi;h_2,h_0)=BC(\phi;h_2,h_1)+BC(\phi;h_1,h_0) ;
    \end{equation*}
    \item $\bp\p BC(\phi;h_1,h_0)=\phi(R(h_1))-\phi(R(h_0)) ;$
    \item If $h_t$ is any smooth path in $\mathcal{H}_E$, we have
    \begin{equation*}
\frac{d}{dt}BC(\phi;h_t,h)=\phi'(R(h_t);\dot{h_t}h_t^{-1}).
    \end{equation*}
\end{enumerate}
\end{prop}

From the construction or property \textit{(2)} in Proposition \ref{cs} and \ref{bc}, it is clear that 
\begin{equation}\label{bccs}
\p BC(\phi;h_1,h_0)=CS(\phi;\nabla^1,\nabla^0),
\end{equation}
where $\nabla^1,\nabla^0$ are Chern connections associated to $h_1, h_0$ respectively. In fact, this is used in \cite[Proposition 3.15]{BC} to prove the property \textit{(2)} for Bott-Chern forms.

As noted in \cite[p84]{BC}, in general, the formula (\ref{homotop}) contains nonlinear terms and cannot be directly integrated to give an explicit formula for Bott-Chern forms. A notable exception to this is when $E$ is a line bundle and $\phi$ is a Chern character polynomial. In this case, we have the following explicit formula for $BC(\ch;h_1,h_0)$. Let $h_1=h_0e^{-\varphi}$ and $\omega=iR(h_0)$. Then $iR(h_1)=\omega_{\varphi}:=\omega+i\p\bp\varphi$ and we have $\ch_j(R(h_0))=\omega^j/j!,\; \ch_j(R(h_1))=\omega_{\varphi}^j/j!$. By direct computation, one can check that
\begin{equation}\label{bcch}
BC(\ch_j;h_1,h_0)=\frac{-i}{j!}\sum^{j-1}_{s=0}\varphi \omega_{\varphi}^s\wedge \omega^{j-1-s}.
\end{equation}
Note that $\omega^j/j!$ and the right-hand side of (\ref{bcch}) make sense for arbitrary $[\omega]\in H^{1,1}(M',\mathbb{R})$. In particular, it makes sense for arbitrary K\"ahler classes that are not necessarily polarized.

Our convention for Bott-Chern forms is identical to the one in \cite{BC,tian1999bott}, differs from the one in \cite{bismut1988analytic} by multiplication of $-i$, and differs from the one in \cite{donaldson1985anti,weinkove2002higher} by multiplication of $i$.

\subsection{Ray-Singer analytic torsion and Quillen anomaly formula}\label{subsec}
In this subsection we introduce the Ray-Singer analytic torsion, the Quillen anomaly formula, and asymptotic expansion of the Ray-Singer analytic torsion only to the extent necessary for this paper. For a more general account of the related results, see \cite{bismut1988analytic,ma2007holomorphic}.

Let $(E,h_E)$ be a Hermitian holomorphic vector bundle over $M$. Let $\bp_E$ be the Dolbeault operator acting on $\Omega^{0,\bullet}(M,E)$ and denote by $\bp^{*}_E$ the formal adjoint of $\bp_E$ with respect to the $L^2$ metric $|\cdot|_{\Omega^{0,\bullet}(M,E)}$ induced by $h_E$ and $\omega^n/n!$. Let $\square_E:=\bp_E\bp^{*}_E+\bp^*_E\bp_E$ be the Kodaira Laplacian on $\Omega^{0,\bullet}(M,E)$ and $e^{-u\square_E}$ be the associated heat operator. By Hodge theory, $\square_E$ has a finite-dimensional kernel. Denote by $P^\bot$ the orthogonal projection onto its orthogonal complement. Define $\zeta$-function by
\begin{equation*}
\zeta_E(z)=-\mathcal{M}[\Tr_s[Ne^{-u\square_E}P^\bot]],
\end{equation*}
where $N$ is the number operator $N:\a\in\Omega^{0,p}(M)\mapsto p\a$, $\Tr_s$ is a supertrace $\Tr[(-1)^N\cdot]$, and $\mathcal{M}[f(u)]$ denotes Mellin transform defined by meromorphic extension of
\begin{equation*}
\mathcal{M}[f(u)](z):=\frac{1}{\Gamma(z)}\int_0^\infty f(u)u^{z-1}du\;\;\text{for}\;\re z \gg 0
\end{equation*}
over $\mathbb{C}$.
\begin{definition}
The Ray-Singer analytic torsion of $(E,h_E)$ is defined as
\begin{equation*}
T(\omega,h_E):=\exp{(-\frac12\zeta'_E(0)))}.
\end{equation*}
In terms of $\zeta$-regularized determinant, 
\begin{equation*}
T(\omega,h_E)=\prod_p {\det}'(\square_E|_{\Omega^{0,p}})^{p(-1)^{p+1}/2},
\end{equation*}
where $'$ means exclusion of zero modes.
\end{definition}

The determinant of the cohomology of $E$ is the complex line given by 
\begin{equation*}
\det H^{\bullet}(M,E)=\bigotimes_{p=0}^{n}(\det H^p(M,E))^{(-1)^p}.
\end{equation*}
Define $\lambda(E):=(\det H^{\bullet}(M,E))^{-1}$. At the level of harmonic representatives, the $L^2$-metric $|\cdot|_{\Omega^{0,\bullet}(M,E)}$ induces the $L^2$-metric $|\cdot|_{\lambda(E)}$ on $\lambda(E)$.
\begin{definition}
The Quillen metric $||\cdot||_{\lambda(E)}$ on the complex line $\lambda(E)$ is defined by 
\begin{equation*}
||\cdot||_{\lambda(E)}=|\cdot|_{\lambda(E)}\times T(\omega,h_E).
\end{equation*}
\end{definition}
The Quillen metric $||\cdot||_{\lambda(E)}$ depends on the K\"ahler metric $\omega$ and $h_E$. The Quillen anomaly formula tells exactly how $||\cdot||_{\lambda(E)}$ changes when $\omega$ and $h_E$ vary.

\begin{thm}[Quillen anomaly formula, \protect{\cite[III, Theorem 1.23]{bismut1988analytic}}]
Let $\omega', \omega$ be the K\"ahler metrics on $M$ and $h', h$ be the Hermitian metrics on $E$. Let $||\cdot||'_{\lambda(E)}, ||\cdot||_{\lambda(E)}$ be the Quillen metrics associated to $(\omega',h'), (\omega,h)$, respectively. Then we have the following.
\begin{equation}\label{anomaly}
(2\pi)^n\log \frac{||\cdot||^{'2}_{\lambda(E)}}{||\cdot||^2_{\lambda(E)}}=i\int_M BC(\Td;\omega',\omega)\ch(R_E(h))+\Td(R_{T^{1,0}M}(\omega'))BC(\ch;h',h).
\end{equation}
\end{thm}

Now let $(L,h)$ be a Hermitian line bundle with $R(h)=-i\omega$. Note that $\omega$ determines $h$ up to a multiplicative constant. In \cite{vasserot1989asymptotics}, Bismut-Vasserot obtained the leading order asymptotics of $\log T(\omega,h^k)$ as $k\rightarrow\infty$, and recently Finski \cite{finski2018full} established the full asymptotic expansion in terms of local coefficients. 

\begin{thm}[\protect{\cite[Theorem 1.1]{finski2018full}}]\label{Finski}
There are local coefficients $\a_j, \b_j\in\mathbb{R}, j\in\mathbb{N}$ such that for any $s\in\mathbb{N}$, as $k\rightarrow\infty$,
\begin{equation*}
-2(2\pi)^n\log T(\omega,h^k)=\sum_{j=0}^{s}k^{n-j}(\a_j\log k + \b_j(\omega)) + o(k^{n-s}).
\end{equation*}
Moreover, the coefficients $\a_j$ do not depend on $\omega$.
\end{thm}

Here, the local coefficient means that it can be expressed as an integral of a density defined locally over $M$. More precisely, $\b_j$ is given by $-\mathcal{M}[\int_M \Tr_s [N{a}_{j,u}(x)]\frac{\omega^n}{n!}]'(0)$ where ${a}_{j,u}$ is given by universal polynomial in $R(\omega)$ and its derivatives \cite[Theorem 1.2, 4.17]{dai2006asymptotic}. Thus, if we change the K\"ahler form $\omega$ by the biholomorphism of $M$, $\b_j$ remains the same. That is, for $f\in Aut(M)$, $\b_j(\omega)=\b_j(f^*\omega)$ where $\b_j$ is understood as a formal expression. From the explicit formula for $\b_0$ \cite[Theorem 8]{vasserot1989asymptotics} and $\b_1$ \cite[Theorem 1.3]{finski2018full}, we observe that $\b_0$ and $\b_1$ only depend on $[\omega]$. For notational convenience, set $\b_{-1}:=0$.

\begin{remark}\label{Finskirmk}
Indeed for $f\in Aut(M)$, one can verify that $\square_{L^k} = (f^{-1})^* \circ\square_{f^* L^k}\circ f^*$, where $\square_{f^* L^k}$ is defined with respect to $f^* \omega$ and $f^* h^k$. Thus, the spectra of the two Laplacians, and consequently the analytic torsions corresponding to $\omega$ and $f^* \omega$, coincide. Since both analytic torsions admit asymptotic expansions by Theorem \ref{Finski}, the coefficients of these expansions, $\b_j(\omega)$ and $\b_j(f^*\omega)$, must be identical. The point of the above argument using the local nature of $\b_j$ is that it does not involve the polarizing line bundle $L$.
\end{remark}

\section{Asymptotic expansion of the DPP partition function}\label{sectiondpp}

We start with the basic definitions regarding the determinantal point processes associated with $(L^k,h^k)$ for each $k\in\mathbb{N}$, with $\omega=iR(h)$ being a K\"ahler form. Fix a basis $(\psi_i^k)_{i=1}^{d_k}$ of $H^0(M,L^k)$. Denote by $\Psi^k$ the holomorphic section of $(L^k)^{\boxtimes d_k}$ over $M^{d_k}$ defined by the Slater determinant
\begin{equation*}
\Psi^k(z_1,\cdots,z_{d_k}):=\frac{1}{\sqrt{d_k}}\det \left(\psi^k_i(z_j)\right)_{i,j}.
\end{equation*}

Physically, $\Psi^k$ represents a collective wave function of free $d_k$ fermions on the lowest Landau level. Note that $\Psi^k$ can be identified as an element of $\det H^0(M,L^k)$, and $L^2$-metrics of $H^0(M^{d_k},(L^k)^{\boxtimes d_k})$ and $\det H^0(M,L^k)$ are related by
\begin{equation}\label{Gram}
||\Psi^k||^2_{H^0(M^{d_k},(L^k)^{\boxtimes d_k})}=\det(\langle \psi^k_i,\psi^k_j \rangle_{H^0(M,L^k)})_{i,j}.
\end{equation}
\begin{definition}
The determinantal point process (DPP) associated to $(L^k, h^k)$ is a probability measure $\mu_{k,\omega}$ on $M^{d_k}$ defined as
\begin{equation*}
\mu_{k,\omega}:=\frac{1}{Z_k[h]}|\Psi^k(z_1,\cdots,z_{d_k})|^2_{(h^k)^{\boxtimes d_k}}\prod_{i=1}^{d_k}\frac{\omega^n}{n!},
\end{equation*}
where $Z_k[h]$ is the normalization constant given by
\begin{equation*}
Z_k[h]=\int_{M^{d_k}}|\Psi^k(z_1,\cdots,z_{d_k})|^2_{(h^k)^{\boxtimes d_k}}\prod_{i=1}^{d_k}\frac{\omega^n}{n!} = \det(\langle \psi^k_i,\psi^k_j \rangle_{H^0(M,L^k)})_{i,j}
\end{equation*}
by virtue of (\ref{Gram}). $Z_k[h]$ is called the partition function of $\mu_{k,\omega}$, and $\log Z_k[h]$ is called the generating functional of $\mu_{k,\omega}$. Note that although $Z_k[h]$ and $|\cdot|^2_{(h^k)^{\boxtimes d_k}}$ depend on $h$, the DPP $\mu_{k,\omega}$ only depends on $\omega$.
\end{definition}

Now fix K\"ahler form $\omega$ in $\cc_1(L)$ and Hermitian metric $h$ on $L$ so that $\omega=iR(h)$. Hermitian metrics $h_\varphi$ on $L$ can be identified with smooth functions $\varphi\in C^\infty(M,\mathbb{R})$ by $h_\varphi:=he^{-\varphi}$. Then the associated Chern curvature form satisfies $iR(h_\varphi)=\omega+i\p\bp\varphi$. From now on, we will use $\varphi$ instead of $h_\varphi$ to denote the Hermitian metrics on $L$ (e.g., $Z[0]=Z[h]$).

Denote by $\mathcal{K}_{\omega}$ the space of K\"ahler potentials for K\"ahler metrics in $\cc_1(L)$. That is,
\begin{equation*}
\mathcal{K}_{\omega}=\{ \varphi\in C^{\infty}(M,\mathbb{R}) : \omega_{\varphi}:=\omega+i\p\bp\varphi>0 \}.
\end{equation*}
By the $\p\bp$-lemma, two K\"ahler potentials define the same metric if and only if they differ by an additive constant. Thus, the space of K\"ahler metrics $\mathcal{K}_0$ on $M$ in $\cc_1(L)$ can be identified with $\mathcal{K}_\omega/\mathbb{R}$, where $\varphi_2\sim\varphi_1$ if and only if $\varphi_2=\varphi_1+c$ for some constant $c$. Note that $\mathcal{K}_{\omega}$, and hence $\mathcal{K}_{0}$, are simply connected.

A different choice of basis $(\psi_i^k)_{i=1}^{d_k}$ changes $\log Z_k[\varphi]$ by the addition of a constant. However, the difference $\log Z_k[\varphi_2]-\log Z_k[\varphi_1]$ is well-defined for two K\"ahler potentials $\varphi_2, \varphi_1$ in $\mathcal{K}_\omega$. We will compute the large $k$ asymptotics of $\log Z_k[\varphi]-\log Z_k[0]$ in two ways, following \cite{klevtsov2017quantum} and \cite{shen2025geometric}.

First, $\log Z_k[\varphi]$ is essentially the $\mathcal{L}$-functional introduced in \cite{donaldson2005scalar}. The following lemma is due to Donaldson: 
\begin{lem}[\protect{\cite[Lemma 2]{donaldson2005scalar}}]
The derivative of $\log Z_k[\varphi]$ on $\mathcal{K}_\omega$ is given by
\begin{equation}\label{varD}
\delta \log Z_k[\varphi] = \int_M \delta\varphi(\Delta_{\varphi}\rho_k(\omega_{\varphi})-k\rho_k(\omega_{\varphi}))\frac{\omega_{\varphi}^n}{n!}.
\end{equation}
\end{lem}
Choose any smooth path $\varphi_t$ in $\mathcal{K}_\omega$ joining $0$ to $\varphi$ and let $\omega_t:=\omega_{\varphi_t}$. We use subscript to denote objects associated with $\omega_t, \omega_\varphi$, etc. Using asymptotic expansion (\ref{Basymp}) to integrate (\ref{varD}), we get a large $k$ asymptotics
\begin{equation*}
(2\pi)^n\log \frac{Z_k[\varphi]}{Z_k[0]}=k^{n+1}\int_0^1 dt\int_M \dot{\varphi_t}(-a_{0}(\omega_t))\frac{\omega_t^n}{n!}+k^{n}\int_0^1 dt\int_M \dot{\varphi_t}(\Delta_t a_{0}(\omega_t)-a_{1}(\omega_t))\frac{\omega_t^n}{n!}+\cdots,
\end{equation*}
where the coefficient of the $k^{n+1-j}$-term is given by 
\begin{equation}\label{1}
\int_0^1 dt\int_M \dot{\varphi_t}(\Delta_t a_{j-1}(\omega_t)-a_{j}(\omega_t))\frac{\omega_t^n}{n!}.
\end{equation}

Alternatively, by the Kodaira-Serre vanishing theorem, the higher cohomology of $L^k$ vanishes for $k\gg1$ and $\log\frac{Z_k[\varphi]}{Z_k[0]}$ becomes a log-ratio of $L^2$-metrics on $\lambda(L^k)$. That is,
\begin{equation}\label{Qu}
\log\frac{Z_k[\varphi]}{Z_k[0]}=\log\frac{|\cdot|^2_{\lambda(L^k),\omega}}{|\cdot|^2_{\lambda(L^k),\omega_\varphi}}=\log\frac{||\cdot||^2_{\lambda(L^k),\omega}}{||\cdot||^2_{\lambda(L^k),\omega_\varphi}}+2\log \frac{T(\omega_\varphi,h^k_\varphi)}{T(\omega,h^k)},
\end{equation}
for $k\gg1$. The large $k$ asymptotics of the first term of the RHS of (\ref{Qu}) is given by the Quillen anomaly formula (\ref{anomaly}), and for the second term it is given by Theorem \ref{Finski}. As a result, we get an asymptotic expansion of $(2\pi)^n\log \frac{Z_k[\varphi]}{Z_k[0]}$ whose $k^{n+1-j}$-term coefficient is given by
\begin{align}\label{2}
&-i\int_M BC(\Td_j;\omega_\varphi,\omega)\frac{\omega^{n+1-j}}{(n+1-j)!}+\Td_{j}(R_\varphi)\frac{-i}{(n+1-j)!}\sum^{n-j}_{s=0}\varphi \omega_{\varphi}^s\wedge \omega^{n-j-s} \nonumber \\
&\qquad\, +\b_{j-1}(\omega)-\b_{j-1}(\omega_\varphi),
\end{align}
using (\ref{bcch}) to express $BC(\ch;\cdot,\cdot)$ explicitly. Here, the integrands of the first and second terms are understood to be zero for $j>n+1, j>n$, respectively. Combining (\ref{1}) and (\ref{2}), we get the following.

\begin{thm}[Asymptotics of the Partition Functions, Version I]
$(2\pi)^n\log \frac{Z_k[\varphi]}{Z_k[0]}$ admits the following form of asymptotic expansion as $k\rightarrow\infty$:
\begin{equation}\label{ver1}
(2\pi)^n\log \frac{Z_k[\varphi]}{Z_k[0]}=k^{n+1}\widetilde{S_0}[\varphi,0]+k^{n}\widetilde{S_1}[\varphi,0]+k^{n-1}\widetilde{S_2}[\varphi,0]+\cdots,
\end{equation}
where $\widetilde{S_j}[\varphi,0]$ is given by (\ref{1}) and (\ref{2}). $\widetilde{S_j}[\cdot,\cdot]:\mathcal{K}_\omega\times\mathcal{K}_\omega\rightarrow\mathbb{R}$ satisfies
\begin{enumerate}
    \item Cocyle identity : for any three K\"ahler potentials $\varphi_2, \varphi_1, \varphi_0$ in $\mathcal{K}_\omega$,
    \begin{align*}
\widetilde{S_j}[\varphi_1,\varphi_0]&=-\widetilde{S_j}[\varphi_0,\varphi_1],\\   
\widetilde{S_j}[\varphi_2,\varphi_0]&=\widetilde{S_j}[\varphi_2,\varphi_1]+\widetilde{S_j}[\varphi_1,\varphi_0];
    \end{align*}
    \item The derivative of $\widetilde{S_j}[\varphi,0]$ on $\mathcal{K}_\omega$ is given by
\begin{equation}\label{firstvariver1}
\delta_\varphi \widetilde{S_j}[\varphi,0]=\int_M \delta\varphi \left( \Delta_\varphi a_{j-1}(\omega_\varphi)-a_{j}(\omega_\varphi) \right) \frac{\omega_\varphi^n}{n!};
\end{equation}
    \item For $j>n+1$, $\widetilde{S_j}[\varphi_2,\varphi_1]$ is an exact cocyle. That is, it can be written as a difference
\begin{equation*}
\widetilde{S_j}[\varphi_2,\varphi_1]=\widetilde{s_j}[\omega_2]-\widetilde{s_j}[\omega_1],
\end{equation*}
for a local functional $\widetilde{s_j}[\cdot] : \mathcal{K}_0\rightarrow\mathbb{R}$ of K\"ahler forms.
\end{enumerate}
\end{thm}

\begin{proof}
\textit{(1)} follows from (\ref{2}) and properties of Bott-Chern forms in Proposition \ref{bc}. More precisely, we only need to check that 
\begin{equation*}
\overline{S_j}[\varphi_1,\varphi_0]:=\int_M BC(\Td_j;\omega_1,\omega_0)\ch_{n+1-j}(\varphi_0)+\Td_{j}(R_1)BC(\ch_{n+1-j};\varphi_1,\varphi_0)
\end{equation*}
satisfies cocyle identity. For $\varphi_2, \varphi_1, \varphi_0$ in $\mathcal{K}_\omega$,
\begin{align*}
&\overline{S_j}[\varphi_2,\varphi_1]+\overline{S_j}[\varphi_1,\varphi_0]\\
=&\int_M BC(\Td_j;\omega_2,\omega_1)\ch_{n+1-j}(\varphi_1)+\Td_{j}(R_2)BC(\ch_{n+1-j};\varphi_2,\varphi_1) \nonumber \\
&\quad +BC(\Td_j;\omega_1,\omega_0)\ch_{n+1-j}(\varphi_0)+\Td_{j}(R_1)BC(\ch_{n+1-j};\varphi_1,\varphi_0)\\
=&\int_M BC(\Td_j;\omega_2,\omega_0)\ch_{n+1-j}(\varphi_0)+\Td_{j}(R_2)BC(\ch_{n+1-j};\varphi_2,\varphi_0) \nonumber \\
&\quad+BC(\Td_j;\omega_2,\omega_1)
\left(  \ch_{n+1-j}(\varphi_1)-\ch_{n+1-j}(\varphi_0) \right)\nonumber\\
&\quad-\left( \Td_{j}(R_2)-\Td_{j}(R_1) \right)BC(\ch_{n+1-j};\varphi_1,\varphi_0)\\
=&\overline{S_j}[\varphi_2,\varphi_0],
\end{align*}
where in the second identity property \textit{(1)} of Proposition \ref{bc} is used, and in the last identity property \textit{(2)} of Proposition \ref{bc} and integration by parts is used. That is,
\begin{equation*}
\int_M BC(\Td_j;\omega_2,\omega_1)\bp\p BC(\ch_{n+1-j};\varphi_1,\varphi_0)
=\int_M \bp\p BC(\Td_j;\omega_2,\omega_1)BC(\ch_{n+1-j};\varphi_1,\varphi_0). 
\end{equation*}

\textit{(2)} follows from (\ref{1}). \textit{(3)} follows from (\ref{2}) and the local nature of $\b_j$ from Theorem \ref{Finski}.
\end{proof}

\begin{remark}\label{remarkcocycle}
The proof shows that for any $\phi\in I_j(n)$, 
\begin{equation*}
\tilde{S}_{\phi}[\varphi_1,\varphi_0]:=\int_M BC(\phi;\omega_1,\omega_0)\frac{\omega_0^{n+1-j}}{(n+1-j)!}+\phi(R_1)\frac{-i}{(n+1-j)!}\sum^{n-j}_{s=0}(\varphi_1-\varphi_0) \omega_{1}^s\wedge \omega_0^{n-j-s}
\end{equation*}
satisfies cocycle identity without polarization assumption on $[\omega]$. Functionals of similar type were considered in \cite{donaldson1985anti,donaldson1987infinite,tian1999bott}. Note that $\tilde{S}_{\phi}$ can be easily modified to become a functional on the space of K\"ahler forms, as we now demonstrate.
\end{remark}

The coefficients $\widetilde{S_j}[\varphi,0]$ for $j\leq n$ are not geometric
functionals in the sense that they depend not only on the K\"ahler forms $\omega_\varphi$, but also on the K\"ahler potentials $\varphi$. But they do so in a trivial manner: for an arbitrary constant $c$, they satisfy
\begin{equation}\label{addconst}
\widetilde{S_j}[\varphi+c,0]=\widetilde{S_j}[\varphi,0]-c\int_M \Td_j(T^{1,0}M)\ch_{n-j}(L).
\end{equation}
This can be easily verified from either (\ref{1}) or (\ref{2}). We have an alternative expression for the asymptotics (\ref{ver1}) where each of the coefficients is given by a geometric functional in the sense that it only depends on K\"ahler forms. That is, we prove Theorem \ref{m1}. Let $V:=\int_M \omega^n/n!$ be a volume of $(M,\omega)$.

\begin{proof}[Proof of Theorem \ref{m1}]
By asymptotic Riemann-Roch-Hirzebruch formula (\ref{aRR2}), we can express (\ref{ver1}) in the following form :
\begin{equation*}
(2\pi)^n\log \frac{Z_k[\varphi]}{Z_k[0]}=kd_{k}(2\pi)^nS_0[\varphi,0]+k^{n}S_1[\varphi,0]+k^{n-1}S_2[\varphi,0]+\cdots,
\end{equation*}
where the coefficients are given by
\begin{equation*}
S_0[\varphi,0]:=\frac{1}{V}\widetilde{S_0}[\varphi,0]
\end{equation*}
and
\begin{equation}\label{defnofSj}
S_j[\varphi,0]:=\widetilde{S_j}[\varphi,0]-\frac{1}{V}\int_M \Td_j(T^{1,0}M)\ch_{n-j}(L) \times\widetilde{S_0}[\varphi,0]
\end{equation}
for $j>0$. By (\ref{addconst}), for an arbitrary constant $c$,
\begin{align*}
&S_j[\varphi+c,0]=\widetilde{S_j}[\varphi+c,0]-\frac{1}{V}\int_M \Td_j(T^{1,0}M)\ch_{n-j}(L) \times\widetilde{S_0}[\varphi+c,0]\\
&=\widetilde{S_j}[\varphi,0]-c\int_M \Td_j(T^{1,0}M)\ch_{n-j}(L)\nonumber\\
&\qquad-\frac{1}{V}\int_M \Td_j(T^{1,0}M)\ch_{n-j}(L) \times\widetilde{S_0}[\varphi,0]+c\int_M \Td_j(T^{1,0}M)\ch_{n-j}(L)\\
&=S_j[\varphi,0]
\end{align*}
for $j>0$. Thus, the functionals $S_j[\cdot,0],\, j>0$ on $\mathcal{K}_\omega$ descend to the functionals on K\"ahler metrics $\mathcal{K}_0=\mathcal{K}_\omega/\mathbb{R}$. Cocyle identity and exact cocycle property for $j>n+1$ follow from corresponding properties of $\widetilde{S_j}$, and \textit{(2)} follows from (\ref{defnofSj}).
\end{proof}

From (\ref{lower}) and (\ref{firstvar}), one can see that $S_0$ is (minus of) the Aubin-Yau functional $I$, and $S_1$ is the Mabuchi functional $\mathcal{M}$ in K\"ahler geometry. It can also be directly checked by comparing the explicit formula for $I, \mathcal{M}$ \cite{tian2012canonical} with (\ref{2}) and (\ref{defnofSj}). For $j=2$, since $\b_1$ only depends on $[\omega]$, $S_2$ is explicitly given by
\begin{align*}
S_2[\omega_\varphi,\omega]=&-i\int_M BC(\Td_2;\omega_\varphi,\omega)\frac{\omega^{n-1}}{(n-1)!}+\Td_{2}(R_\varphi)\frac{-i}{(n-1)!}\sum^{n-2}_{s=0}\varphi \omega_{\varphi}^s\wedge \omega^{n-2-s}\nonumber\\
&+\frac{i}{V}\left(\int_M \Td_2(T^{1,0}M)\frac{\omega^{n-2}}{(n-2)!}\right)\times\int_M \frac{-i}{(n+1)!}\sum^{n}_{s=0}\varphi \omega_{\varphi}^s\wedge \omega^{n-s}.
\end{align*}
By (\ref{lower}), its first variation is given by
\begin{equation}\label{firstliouville}
\delta_\varphi S_2[\omega_\varphi,\omega]=\int_M \delta\varphi\left(\widehat{a_2(\omega_\varphi)}+\frac{1}{6}\Delta_\varphi S_\varphi-\frac{1}{24}\left(|R_\varphi|^2-4|\ric_\varphi|^2+3S_\varphi^2\right) \right) \frac{\omega_\varphi^n}{n!}.
\end{equation}
When $n=1$ this reduces to
\begin{equation*}
\delta_\varphi S_2[\omega_\varphi,\omega]=\frac16\int_M \delta\varphi\Delta_\varphi R_\varphi \,\omega_\varphi,
\end{equation*}
which is the variation of the Liouville action restricted to fixed area metrics, written in terms of K\"ahler gauge \cite[(5.10)]{klevtsov2014random}. Therefore, $S_2$ extends the Liouville action to all dimensions. We will compute its second variation in Appendix \ref{appendix}.

\begin{remark}\label{klev}
In \cite[(6.6)]{klevtsov2014random}, Klevtsov explicitly computed the asysmptotics of $\log \frac{Z_k[\varphi]}{Z_k[0]}$ up to $j\leq 4$ in dimension $n=1$. Based on that, it was conjectured that for the general dimension $n$, the first $n+2$ coefficients are nontrivial action functionals (satisfying cocycle identity) and terms with $j>n+1$ are exact cocycles. Thus, Theorem \ref{m1} confirms this conjecture. It was also asked what is the relationship between $S_2$ and known functionals in K\"ahler geometry, such as Chen-Tian functional $E_1$ \cite{chen2002ricci}. In fact, from (\ref{firstliouville}), one can see that $S_2$ is a linear combination of the Chen-Tian functional $E_1$ and the Bando-Mabuchi functional $\mathcal{M}_2$ \cite{bando1986some}. By the results of \cite{bando1986some,chen2002ricci}, one can observe that $S_2$ integrates a Futaki-type holomorphic invariant. We will generalize this for all $j>0$ in the next section. 
\end{remark}

Note that (\ref{1}) and (\ref{2}) make sense as formal expressions for a nonpolarized K\"ahler class $[\omega]$. We expect that they are equal in general, without a polarization assumption. In particular, it would imply that the 1-forms defined on $\mathcal{K}_\omega$ by
\begin{equation}\label{1forms}
\gamma^{(j)}_{\varphi}(\psi):=\int_M \psi(\Delta_\varphi a_{j-1}(\omega_\varphi)-a_{j}(\omega_\varphi))\frac{\omega_\varphi^n}{n!},
\end{equation}
for $\psi\in T_{\varphi}\mathcal{K}_\omega=C^{\infty}(M,\mathbb{R})$, are closed. It is known that this holds for $j=0,1$ by \cite{mabuchi1986k}, and for $j=2$ by \cite{bando1986some,chen2002ricci}. We will verify these claims for $j=2$ in Appendix \ref{appendix} by direct computation.

\section{Holomorphic invariants from Bergman kernel asymptotics and critical points}\label{sectionFutaki}

In this section, we show that there is an obstruction to the existence of the critical points for each $S_j$, given by a holomorphic invariant introduced by Futaki \cite{futaki2004asymptotic}. Proof of Theorem \ref{m3} is divided into two part, Theorem \ref{futbar} and Theorem \ref{mainthm}. Denote by $\mathfrak{h}(M)$ the Lie algebra of holomorphic vector fields on $(M,\omega)$. For $X\in\mathfrak{h}(M)$, by Hodge theory, there is a unique harmonic $(0,1)$-form $\tau$ and function $\theta_X$ (called holomorphy potential) satisfying
\begin{equation}\label{hamilton}
\iota_X \omega = \tau -\bp \theta_X.
\end{equation}
Then $\theta_X$ is defined up to an additive constant, and $\tau$ will be assumed to be zero without loss of generality for our purposes. 

If $\theta_X$ is real valued, $\re X$ is a Hamiltonian vector field and $\frac{1}{2}\theta_X$ is a Hamiltonian function with respect to $\omega$, since
\begin{equation*}
\iota_{\re X}\omega=-\frac{1}{2}(\bp \theta_X + \p \overline{\theta_X})=-\frac{1}{2}d\theta_X.
\end{equation*}
That is, $\re X$ lies in a Lie algebra of Hamiltonian symplectomorphisms $\mathcal{G}$ of $(M,\omega)$. 

Consider $\nabla X = X^q_p \frac{\p}{\p z^q}\otimes dz^p$ as an $\End(T^{1,0}M)$-valued $0$-form on $M$, where
\begin{equation*}
X^q_p=\frac{\p X^q}{\p z^p}+\Gamma^{q}_{pr}X^r.
\end{equation*}
Denote by $\flat,\sharp$ the lowering and raising index operations by K\"ahler metric. For example, 
\begin{equation*}
(\overline{\nabla X})^{\flat,\sharp}=g_{p\br}\overline{X^{r}_{l}}g^{q\bl}\frac{\p}{\p z^q}\otimes dz^p.
\end{equation*}
Let $R$ be the curvature 2-form of $\omega$. We introduce the following notation:
\begin{equation*}
\overline{R}:=(R)^{\flat,\sharp}=R^{\bar{q}}{}_{\bar{p}r\bl}dz^r\wedge d\bar{z}^l,
\end{equation*}
which is $\End(T^{0,1}M)$-valued $2$-form on $M$.

First we identify the holomorphic invariants we will be dealing with. We begin with the lemma.
\begin{lem}\label{lulemm}
We have
\begin{equation}\label{nabx}
\iota_X R = -\bp \nabla X
\end{equation}
and 
\begin{equation}\label{nabx2}
\iota_{\overline{X}} \overline{R} = \p \overline{\nabla X}.
\end{equation}

\end{lem}
\begin{proof}
In local coordinates, 
\begin{align*}
\bp \nabla X &= \bp \left(X^q_p\right) = \frac{\p X^q_p}{\p \bz^l} d\bz^l 
=\left( \frac{\p^2 X^q}{\p z^p \p \bz^l} + \frac{\p \Gamma^{q}_{pr}}{\p \bz^l}X^r + \Gamma^{q}_{pr}\frac{\p X^r}{\p \bz^l} \right)d\bz^l\\
&=\frac{\p \Gamma^{q}_{pr}}{\p \bz^l}X^r d\bz^l = -R_{p}{}^{q}{}_{r\bl}X^rd\bz^l=-\iota_X R.
\end{align*}
Similarly, 
\begin{equation*}
\p \overline{\nabla X}=\p \left(\overline{X^q_p}\right)=\frac{\p \overline{\Gamma^{q}_{pl}}}{\p z^r}\overline{X^l} dz^r=\overline{-R_{p}{}^{q}{}_{l\br}X^l}dz^r=-R^{\bar{q}}{}_{\bar{p}r\bl}\overline{X^l} dz^r = \iota_{\overline{X}} \overline{R}
\end{equation*}
where we used the identity 
\begin{equation*}
\overline{R_{i\bj k\bl}}=R_{j\bi l\bk}.
\end{equation*}
\end{proof}

Using the lemma, we can prove the invariance of the following invariants.

\begin{thm}\label{futbar}
Under the same notations as above, for $j\geq 0$, 
\begin{equation}\label{FX}
\widetilde{F_j}(\omega,X):=\int_M \Td_j(R+\nabla X)(\omega+\theta_X)^{n+1-j} 
\end{equation}
and
\begin{equation}\label{FbarX}
\widetilde{F_j}(\omega,\overline{X}):=\int_M \Td_j(R+(\overline{\nabla X})^{\flat,\sharp})(\omega-\overline{\theta_X})^{n+1-j} 
\end{equation}
are independent of the choice of the K\"ahler metric in $\cc_1(L)$, under the normalization of $\theta_X$ by $\int_M \theta_X \omega^n =0$. 
\end{thm}

\begin{proof}
We first prove the claim for (\ref{FX}). The proof for (\ref{FbarX}) is exactly the same, as we will explain.

Let $\omega_t=\omega+i\p\bp \varphi_t$ be an arbitrary one-parameter family of K\"ahler forms in $[\omega]$ with $\omega_0=\omega$. We will show that $\frac{d}{dt}\widetilde{F_j}(\omega_t,X)=0$. Let $\theta_t$ be a holomorphy potential of $X$ with respect to $\omega_t$, that is, $\iota_X \omega_t = -\bp \theta_t$. Setting $\a_t:=-i\p\dot{\varphi_t}$, we have $\dot{\omega_t}=i\p\bp\dot{\varphi_t}=\bp \a_t$.
Since 
\begin{equation*}
\iota_X \omega_t=\iota_X \omega + i\iota_X\p\bp\varphi_t=-\bp\theta_0+i\bp\iota_X (\p\varphi_t), 
\end{equation*}
we have $\theta_t=\theta_0-i\iota_X \p\varphi_t$ and thus $\dot{\theta_t}=\iota_X \a_t$ up to an additive constant.
We used the fact that $X\in\mathfrak{h}(M)$. Also, by torsion freeness of the Levi–Civita connection we have
\begin{equation*}
\dot{\left(\nabla X \right)}=\dot{X^q_p}=\dot{\Gamma^{q}_{pr}}X^r= \iota_X \dot{\Gamma_t},
\end{equation*}
where $\Gamma_t$ denotes the Levi–Civita connection $1$-form for $\omega_t$. Note that $\dot{\Gamma_t}$ is globally well defined and $\dot{R_t}=\bp\dot{\Gamma_t}$. Using these, we compute $\frac{d}{dt}\widetilde{F_j}(\omega_t,X)$ (we suppress the subscript $t$):  
\begin{align}\label{5}
&\frac{d}{dt}\widetilde{F_j}(\omega_t,X)=\int_M j\Td_j (\dot{\nabla X}+\dot{R}, \nabla X + R,\cdots)(\omega+\theta)^{n+1-j}\nonumber\\
&\qquad\qquad\qquad\quad+(n+1-j)\Td_j(\nabla X + R)(\dot{\theta}+\dot{\omega})(\omega+\theta)^{n-j} \nonumber\\
&=\int_M j\Td_j (\iota_X \dot{\Gamma}+\bp\dot{\Gamma}, \nabla X + R,\cdots)(\omega+\theta)^{n+1-j} \nonumber \\
&\qquad +(n+1-j)\Td_j(\nabla X + R)(\iota_X \a+\bp \a)(\omega+\theta)^{n-j}\nonumber\\
&=\int_M j\Td_j (\iota_X \dot{\Gamma}, \nabla X + R,\cdots)(\omega+\theta)^{n+1-j} \nonumber\\
&\qquad+j(j-1)\Td_j (\dot{\Gamma}, \bp\nabla X,\nabla X + R,\cdots)(\omega+\theta)^{n+1-j} \nonumber\\
&\qquad +j(n+1-j)\Td_j (\dot{\Gamma},\nabla X + R,\cdots)\bp\theta(\omega+\theta)^{n-j} \nonumber\\
&\qquad +(n+1-j)\Td_j(\nabla X + R)\iota_X \a (\omega+\theta)^{n-j} \nonumber\\
&\qquad -(n+1-j)\bp\Td_j(\nabla X + R) \a (\omega+\theta)^{n-j} \nonumber\\
&\qquad +(n+1-j)(n-j)\Td_j(\nabla X + R) \a \bp\theta(\omega+\theta)^{n-1-j},
\end{align}
where to get the last identity we used $\bp R=\bp\omega=0$ and integration by parts.

Now by definition of $\theta$ and Lemma \ref{lulemm}, we have
\begin{equation}\label{lem1}
\iota_X (\omega+\theta)=\iota_X \omega = -\bp\theta
\end{equation}
and
\begin{equation}\label{lem2}
\iota_X (\nabla X+ R)=\iota_X R = -\bp\nabla X.
\end{equation}
Using these, we compute
\begin{align}\label{6}
&\int_M \iota_X \left[j\Td_j(\dot{\Gamma},\nabla X+R,\cdots)(\omega+\theta)^{n+1-j}+(n+1-j)\Td_j(\nabla X+R)\a(\omega+\theta)^{n-j} \right]\nonumber\\
&=\int_M j\Td_j(\iota_X \dot{\Gamma},\nabla X+R,\cdots)(\omega+\theta)^{n+1-j}\nonumber\\
&\qquad-j(j-1)\Td_j(\dot{\Gamma},-\bp\nabla X,\nabla X+R,\cdots)(\omega+\theta)^{n+1-j}\nonumber\\
&\qquad -j(n+1-j)\Td_j (\dot{\Gamma},\nabla X+R,\cdots)(-\bp\theta)(\omega+\theta)^{n-j} \nonumber\\
&\qquad +j(n+1-j)\Td_j (-\bp\nabla X,\nabla X+R,\cdots)\a(\omega+\theta)^{n-j} \nonumber\\
&\qquad +(n+1-j) \Td_j (\nabla X+R) \iota_X \a (\omega+\theta)^{n-j} \nonumber\\
&\qquad -(n+1-j)(n-j)\Td_j (\nabla X+R) \a (-\bp\theta)(\omega+\theta)^{n-1-j}.
\end{align}
Since $\bp \Td_j(\nabla X + R)=j\Td_j(\bp\nabla X,\nabla X + R,\cdots)$, (\ref{5}) and (\ref{6}) are equal.

Hence we obtain for some differential form $\eta$,
\begin{equation*}
\frac{d}{dt}\widetilde{F_j}(\omega_t,X)=\int_M \iota_X \eta =0
\end{equation*}
by dimensional reason.

We now turn to the proof for (\ref{FbarX}). Note that 
\begin{equation*}
\Td_j(R+(\overline{\nabla X})^{\flat,\sharp})=\Td_j(\overline{R}+\overline{\nabla X}),
\end{equation*}
where on the right hand side trace is taken as $\End(T^{0,1}M)$.
Then as the proof for (\ref{FX}) shows, we only need to check the following identities:
\begin{equation*}
\dot{\overline{R}}=-\p \dot{\overline{\Gamma}};\quad \dot{\overline{\nabla X}}=\iota_{\overline{X}} \dot{\overline{\Gamma}};\quad \dot{\omega}=\p \overline{\alpha};\quad \dot{\overline{\theta}}=\iota_{\overline{X}}\overline{\alpha},
\end{equation*}
and
\begin{equation*}
\iota_{\overline{X}} (\omega-\overline{\theta})= -\p\overline{\theta};\quad 
\iota_{\overline{X}} (\overline{R}+\overline{\nabla X})=\p \overline{\nabla X}.
\end{equation*}
For example,
\begin{equation*}
\dot{\overline{R}}=\dot{R}^{\bar{q}}{}_{\bar{p}r\bl}dz^r\wedge d\bar{z}^l=\dot{\overline{R_{p}{}^{q}{}_{l \overline{r} } }}dz^r\wedge d\bar{z}^l=\overline{ -\p_{\br} \dot{\Gamma^{q}_{pl}}}dz^r\wedge d\bar{z}^l=-\p_r \dot{\overline{\Gamma^{q}_{pl}}}dz^r\wedge d\bar{z}^l=-\p\dot{\overline{\Gamma}}.
\end{equation*}
The proof then goes without changes, and one can show that the time derivative is equal to 
\begin{equation*}
\int_M \iota_{\overline{X}} \left[j\Td_j(\dot{\overline{\Gamma}},\overline{\nabla X}+\overline{R},\cdots)(\omega-\overline{\theta})^{n+1-j}-(n+1-j)\Td_j(\overline{\nabla X}+\overline{R})\overline{\a}(\omega-\overline{\theta})^{n-j} \right],
\end{equation*}
which is zero. Note the minus sign in front of the second term.
\end{proof}

\begin{remark}
The proof shows that for any $\phi\in I_j(n)$,
\begin{equation*}
\int_M \phi(R+\nabla X)(\omega+\theta_X)^{n+1-j}
\end{equation*}
is independent of the choice of the K\"ahler metric in $[\omega]$ without a polarization assumption on $[\omega]$. They are precisely holomorphic invariants introduced by Futaki \cite{futaki2004asymptotic}, generalizing various integral invariants, including the Futaki invariant \cite{futaki1983compact} and the Bando-Futaki invariant \cite{bando2006obstruction}. As noted before, it is not invariant under the addition of a constant to $\theta_X$, but it can be easily modified to be invariant. For example,
\begin{equation*}
F_j(\omega,X):=-i\int_M \Td_j(R+\nabla X)\frac{(\omega+\theta_X)^{n+1-j}}{(n+1-j)!}+\widehat{a_j(\omega_\varphi)}\times i\int_M\frac{(\omega+\theta_X)^{n+1}}{(n
+1)!}
\end{equation*}
is invariant under the addition of a constant to $\theta_X$.
\end{remark}

Now we prove the rest of Theorem \ref{m3}. It generalizes formula \cite[(4.4)]{lu2004k} for all $j\geq 0$. 

\begin{thm}\label{mainthm}
Let $(M,\omega)$ be a polarized K\"ahler manifold. Let $X\in\mathfrak{h}(M)$ and $\theta_X$ be as in (\ref{hamilton}) and suppose it is purely imaginary. Then for all $j\geq 0$, we have the following identity:
\begin{equation}\label{mainformula}
\int_M \theta_X \left( a_j(\omega)-\Delta a_{j-1}(\omega) \right) \frac{\omega^n}{n!}= \frac{1}{(n+1-j)!}\int_M \Td_j(R+\nabla X)(\omega+\theta_X)^{n+1-j}.
\end{equation}
More generally, for non purely imaginary $\theta_X$, we have
\begin{equation}\label{mainformula2}
\int_M i\Im \theta \left( a_j(\omega)-\Delta a_{j-1}(\omega) \right) \frac{\omega^n}{n!}= \frac{1}{(n+1-j)!}\int_M \Td_j(R+\frac12(\nabla X +(\overline{\nabla X})^{\flat,\sharp}))(\omega+i\Im \theta)^{n+1-j}.
\end{equation}
\end{thm}

\begin{remark}
It is clear that the right hand side of (\ref{mainformula}) is $\frac{\widetilde{F_j}(\omega,X)}{(n+1-j)!}$, and the right hand side of (\ref{mainformula2}) is nothing but $\frac{\widetilde{F_j}(\omega,X)}{2(n+1-j)!}+\frac{\widetilde{F_j}(\omega,\overline{X})}{2(n+1-j)!}$; see (\ref{4}). Thus by Theorem \ref{futbar} and Theorem \ref{mainthm}, we obtain holomorphic invariants out of coefficients of Bergman kernel asymptotic expansion.
\end{remark}

\begin{proof}
We start with the purely imaginary $\theta_X$ case. Let $f_t \in Aut(M)$ be the flow of $\re X=(X+\overline{X})/2$. Let $\omega_t:=f_t^*\omega$. Then
\begin{equation}\label{lie}
\dot{\omega_t}|_{t=0}=\mathcal{L}_{\re X}\omega=\frac{1}{2}\left(\p\iota_X \omega +\bp\iota_{\overline{X}}\omega\right)=\frac{1}{2}\left(-\p\bp\theta_X-\bp\p\overline{\theta_X}\right)=-\p\bp\theta_X.
\end{equation}
Since $\omega_t$ has the same K\"ahler class $[\omega]$, we have $\omega_t=\omega+i\p\bp\varphi_t$  with
$\dot{\varphi_t}|_{t=0}=i\theta_X$ upto constant.

Consider $\left.-\frac{d}{dt}\right|_{t=0}\widetilde{S_j}[\varphi_t,0]$. From the first variation formula (\ref{firstvariver1}), we have
\begin{align}\label{integrates}
\left.-\frac{d}{dt}\right|_{t=0}&\widetilde{S_j}[\varphi_t,0]=\int_M \dot{\varphi_t}\left( a_j(\omega_t)-\Delta_t a_{j-1}(\omega_t) \right) \frac{\omega_t^n}{n!}=i\int_M \theta_X\left( a_j(\omega)-\Delta a_{j-1}(\omega) \right) \frac{\omega^n}{n!}. 
\end{align}

Alternatively, from (\ref{2}) we have
\begin{align*}
\left.-\frac{d}{dt}\right|_{t=0}\widetilde{S_j}[\varphi_t,0]&=i\int_M \left.\frac{d}{dt}\right|_{t=0}BC(\Td_j;\omega_t,\omega)\frac{\omega^{n+1-j}}{(n+1-j)!}\nonumber\\
&\qquad+\Td_{j}(R)\left.\frac{d}{dt}\right|_{t=0}\frac{-i}{(n+1-j)!}\sum^{n-j}_{s=0}\varphi_t \omega_{t}^s\wedge \omega^{n-j-s},
\end{align*}
where we used the fact that $\frac{d}{dt}\left(\b_{j-1}(f_t^*\omega)-\b_{j-1}(\omega)\right)=0$ and $\varphi_0=0$. 

Now we compute $\frac{d}{dt}BC(\cdot;\omega_t,\omega)$ using property \textit{(3)} of Bott-Chern forms (Proposition \ref{bc}).
First, by (\ref{lie}) and (\ref{hamilton}), we have
\begin{equation*}
\left.\dot{\omega_t}\right|_{t=0}=-\p\bp\theta_X=\p(\iota_X \omega).
\end{equation*}
This implies, by local computation (we are identifying $\omega$ and the Hermitian metric it defines on $T^{1,0}M$),
\begin{equation*}
\left.\dot{\omega_t}{\omega_t}^{-1}\right|_{t=0}=\frac{\p}{\p z^p} \left( X^s g_{s\bk} \right) g^{q\bk}=\frac{\p X^s}{\p z^p}\delta^q_s+X^s\frac{\p g_{s\bk}}{\p z^p}g^{q\bk}=\frac{\p X^q}{\p z^p}+X^s\Gamma^{q}_{sp}=X^q_p=\nabla X.
\end{equation*}
Thus we have
\begin{equation}\label{Tddiff}
\left.\frac{d}{dt}\right|_{t=0}BC(\Td_j;\omega_t,\omega)=\Td'_j(R;\nabla X).
\end{equation}
For the $\left.\frac{d}{dt}\right|_{t=0}BC(\ch_{n+1-j};\omega_t,\omega)$ part, we can directly differentiate to get
\begin{equation}\label{chdiff}
\left.\frac{d}{dt}\right|_{t=0}\frac{-i}{(n+1-j)!}\sum^{n-j}_{s=0}\varphi_t \omega_{t}^s\wedge \omega^{n-j-s}=\frac{-i}{(n+1-j)!}\sum^{n-j}_{s=0}\dot{\varphi_t}|_{t=0}\omega^{n-j}=\frac{\theta_X \omega^{n-j}}{(n-j)!}.
\end{equation}
Combining (\ref{Tddiff}) and (\ref{chdiff}), we obtain
\begin{equation}\label{3}
\left.-\frac{d}{dt}\right|_{t=0}\widetilde{S_j}[\varphi_t,0]=i\int_M j\Td_j (\nabla X,R,\cdot\cdot\cdot,R)\frac{\omega^{n+1-j}}{(n+1-j)!}+\Td_j(R,\cdot\cdot\cdot,R)\frac{\theta_X \omega^{n-j}}{(n-j)!}.
\end{equation}
On the other hand, we have
\begin{align}\label{4}
\int_M \Td_j(R+\nabla X)(\omega+\theta_X)^{n+1-j}&=\int_M \Td_j(R,\cdot\cdot\cdot,R)(n+1-j)\theta_X \omega^{n-j} \nonumber\\
&\qquad+j\Td_j(\nabla X,R,\cdot\cdot\cdot,R)\omega^{n+1-j}
\end{align}
by dimensional reason. Comparing (\ref{3}) and (\ref{4}), we conclude
\begin{equation*}
\left.-\frac{d}{dt}\right|_{t=0}\widetilde{S_j}[\varphi_t,0]=\frac{i}{(n+1-j)!}\int_M \Td_j(R+\nabla X)(\omega+\theta_X)^{n+1-j},
\end{equation*}
which completes the proof. 

For the general case, note that
\begin{equation*}
\dot{\omega}=\mathcal{L}_{\re X}\omega=-\p\bp\left(\frac{\theta_X-\overline{\theta_X}}{2}\right)=-\p\bp i\Im \theta_X 
\end{equation*}
and
\begin{equation*}
\dot{\omega}{\omega}^{-1}=\frac{\p(\iota_X \omega)+\bp \iota_{\overline{X}} \omega}{2}{\omega}^{-1}=\frac12(\nabla X +(\overline{\nabla X})^{\flat,\sharp}).
\end{equation*}
The rest of the proof goes without changes.
\end{proof}

\begin{remark}
Note that $\theta_X$ is determined only up to an additive constant, but (\ref{mainformula}) behaves correctly under the addition of a constant to $\theta_X$. Using $S_j$ instead of $\widetilde{S_j}$ in the proof, a similar formula can be obtained which does not depend on the normalization of $\theta_X$ for $j>0$ (that is, $F_j(\omega,X)$ instead of $\widetilde{F_j}(\omega,X)$ on the right hand side). 
\end{remark}

We showed that the non-vanishing of $F_j$ obstructs the existence of the critical points of $S_j$. That is, for $X\in\mathfrak{h}(M)$ with purely imaginary holomorphy potential and $f_t \in Aut(M)$ be the flow of $\re X=(X+\overline{X})/2$,
\begin{equation*}
\frac{d}{dt}S_j[f^*_t\omega,\omega]=F_j(\omega,X).
\end{equation*}

When $\mathfrak{h}(M)=0$, the obstruction by $F_j$ becomes trivial. In this regard, assuming that $Aut(M,L)$ is discrete, we can prove Proposition \ref{m4} using Donaldson's result on balanced metrics. Here we call $\omega'$ balanced at level $k$ if the function $\rho_k(\omega')$ is constant.  

\begin{proof}[Proof of Proposition \ref{m4}]
We prove this by induction on $j\geq 1$. Suppose that $\omega_\infty$ is a critical point of $S_j$ for all $1\leq j\leq m$. In particular, $\omega_\infty$ has a constant scalar curvature. By the proof of \cite[Theorem 3]{donaldson2001scalar}, there is a sequence of K\"ahler metrics $\omega_k$ balanced at level $k$ for large enough $k$, such that $||\omega_k-\omega_\infty||_{C^r(M,\omega_\infty)}=O(k^{-q})$ for arbitrary $r,q$. See also \cite[Section 4.3]{donaldson2001scalar}. Choose $r\geq 2m$ and $q\geq m+1$. By the induction hypothesis, we have $a_j(\omega_\infty)-\Delta a_{j-1}(\omega)\equiv\text{constant}=\widehat{a_j}$ for $j\leq m$, which implies $a_j(\omega_\infty)\equiv\widehat{a_j}$ for $j\leq m$. Note that the Bergman kernel asymptotics (\ref{Basymp}) is uniform in the sense that there is a fixed constant $C$ such that
\begin{equation}\label{uniform}
\left\Vert (2\pi)^n\rho_k(\omega_k)-\sum^{m+1}_{j=0} a_j(\omega_k)k^{n-j} \right\Vert_{C^0} \leq Ck^{n-m-2}
\end{equation}
for all $k\gg1$. Since $\omega_k$ are balanced at level $k$, we have
\begin{equation}\label{balanced}
(2\pi)^n\rho_k(\omega_k)\equiv(2\pi)^n \frac{d_k}{V} =  k^n+\widehat{a_1}k^{n-1}+\widehat{a_2}k^{n-2} +\cdots.
\end{equation}
Substituting (\ref{balanced}) in (\ref{uniform}) with the induction hypothesis, we get
\begin{equation*}
\left\Vert \widehat{a_{m+1}}-a_{m+1}(\omega_k) \right\Vert_{C^0} \leq C'k^{-1},
\end{equation*}
for some constant $C'$. Since $\omega_k\rightarrow\omega_\infty$ in $C^\infty$, $a_{m+1}(\omega_\infty)$ is constant and hence $\delta S_{m+1}[\omega_\infty,\omega]=\int_M \delta \varphi_\infty(\widehat{a_{m+1}}+\Delta a_{m}(\omega_\infty)-a_{m+1}(\omega_\infty))\omega_\infty^n/n!=0 $. This completes the induction step.
\end{proof}

Finally, by expression (\ref{mainformula}), we can derive a Bott-type residue formula for the LHS of (\ref{mainformula}) (modified to be invariant under the normalization of $\theta_X$, if necessary). As noted in \cite[Theorem 6.3]{tianresidue}, the proof of that theorem applies directly to our case as well. More precisely, we only need to check that $\bp[ \Td_j(R+\nabla X)(\omega+\theta_X)^{n+1-j}]= -\iota_X [\Td_j(R+\nabla X)(\omega+\theta_X)^{n+1-j}]$, which is immediate from (\ref{lem1}) and (\ref{lem2}). See also \cite[Theorem 5.2.8]{futaki2006kahler}. As the corresponding modification of the statement is routine, we do not include it here.

\section{Non-perturbative approach to the gravitational path integral}\label{sectionpathint}

In \cite{ferrari2012gravitational}, it was shown that the effective action for 2D quantum gravity coupled to nonconformal matter contains $S_1$ and $S_2$, namely, the Mabuchi and Liouville actions. The corresponding string susceptibility was computed at one-loop order in \cite{bilal20142d} by perturbing $S_1$ and $S_2$ around their critical points. For dimension $
n=1$, the critical points of both $S_1$ and $S_2$ correspond to constant curvature metrics, which always exist by the uniformization theorem. However, we have shown that in higher dimensions, there is a nontrivial obstruction to the existence of critical points for each $S_j$. In this section, we briefly review the nonperturbative approach to the gravitational path integral on polarized Kähler manifolds proposed in \cite{ferrari2013random}.

Let $\mathcal{B}_k$ be the set of all Hermitian metrics $H$ on the vector space $H^0(M,L^k)$. There are natural maps between $\mathcal{B}_k$ and $\mathcal{K}_\omega$ :
\begin{equation*}
Hilb_k : \mathcal{K}_\omega \rightarrow \mathcal{B}_k,\quad FS_k : \mathcal{B}_k \rightarrow \mathcal{K}_\omega,
\end{equation*}
defined by
\begin{equation*}
||S||^2_{Hilb_k(\varphi)}:=\frac{d_k}{V}\int_M |S|^2_{h^ke^{-k\varphi}} \frac{\omega_\varphi^n}{n!},\quad FS_k(H):=\frac{1}{k}\log \left(\sum_{i=1}^{d_k}|S^H_i|^2_{h^k} \right)
\end{equation*}
where $(S^H_i)^{d_k}_{i=1}$ is any orthonormal basis of $H^0(M,L^k)$ with respect to $H$. For any $\varphi \in \mathcal{K}_\omega$, let $\varphi_k:=FS_k\circ Hilb_k(\varphi)$. By Tian-Ruan \cite{tian1990set,ruan1996canonical}, $\omega_{\varphi_k}$ converges to $\omega_\varphi$ in $C^\infty$ as $k\rightarrow\infty$. Thus, any K\"ahler metric in $\cc_1(L)$ can be approximated by a metric in the image of $\mathcal{B}_k$ under $FS_k$. In fact, the space $\mathcal{B}_k$ approximates the space $\mathcal{K}_\omega$ in a stronger sense, where the geodesics of $\mathcal{B}_k$ converge to the geodesics of $\mathcal{K}_\omega$ with respect to the natural Riemannian structures. For a more detailed exposition of the subject, see \cite{phong2007lectures,berman2011bergman}. Based on this observation, Ferrari-Klevtsov-Zelditch \cite{ferrari2013random} proposed a formal definition of the path integral on the space of K\"ahler metrics as a limit of finite-dimensional integral over $\mathcal{B}_k/\mathbb{R}$, that is,
\begin{equation*}
\int_{\mathcal{K}_0} \mathcal{O}(\varphi)e^{-S(\varphi)}\mathcal{D}\varphi :=\lim_{k\rightarrow\infty}\int_{\mathcal{B}_k/\mathbb{R}} \mathcal{O}_k(H)e^{-S_k(H)}\mathcal{D}H.
\end{equation*}

See \cite{klevtsov2014stability} for a more detailed analysis of these integrals. Note that for a desired action $S$, one has to find an appropriate sequence of actions $S_k$ on $\mathcal{B}_k$ approximating $S$. Following Donaldson \cite{donaldson2005scalar}, Klevtsov \cite[(7.7)]{klevtsov2014random} defined the functionals $S_{L,k}$ on $\mathcal{B}_k$ that approximate the Liouville action. The following proposition verifies the slight modification of that construction in arbitrary dimension $n$.

\begin{prop}
Choose an orthonormal basis $(\psi_i^k)_{i=1}^{d_k}$ of $H^0(M,L^k)$ with respect to the $L^2$-metric induced from $\omega$. Define the determinant $\det\nolimits_\omega$ on $\mathcal{B}_k$ with respect to $(\psi_i^k)_{i=1}^{d_k}$. Define $S_{L,k}$ on $\mathcal{B}_k$ by
\begin{equation*}
S_{L,k}(H):=\left( (2\pi)^n\log \det\nolimits_{\omega} (H)-(2\pi)^n d_k\log d_k/V - k^nS_1[ \omega_{FS_k(H)},\omega] \right)k^{1-n}.
\end{equation*}
Then $S_{L,k}$ approximates the functional $S_2$ in the following sense.
For any K\"ahler metric $\omega_\varphi \in \mathcal{K}_0$, choose $\varphi$ so that $S_0[\varphi,0]=0$ without loss of generality (in fact, it is customary to identify $\mathcal{K}_0 = \mathcal{K}_\omega/\mathbb{R}$ with $I^{-1}(0)=S_0[\cdot,0]^{-1}(0)$). As $k\rightarrow\infty$, we have
\begin{equation*}
S_{L,k}(Hilb_k(\varphi)) \rightarrow S_2[\omega_\varphi,\omega]
\end{equation*}
uniformly over bounded subsets in $\mathcal{K}_0$.
\end{prop}
\begin{proof}
Let $\varphi_k:=FS_k\circ Hilb_k (\varphi)$. By definition,
\begin{equation*}
\varphi_k= \varphi + \frac{1}{k}\log \rho_k(\omega_\varphi) -\frac{1}{k}\log \frac{d_k}{V}=\varphi +\frac1k\log \frac{\rho_k(\omega_\varphi)}{d_k/V}.
\end{equation*}
From (\ref{Basymp}) and (\ref{aRR}), one can see that $\left\Vert 1- \frac{\rho_k(\omega_\varphi)}{d_k/V}\right\Vert_{C^0}=O(k^{-1})$. Hence we have
\begin{equation}\label{k2}
|| \varphi_k -\varphi ||_{C^0} \leq Ck^{-2}
\end{equation}
for some constant $C$ uniform over bounded subsets in $\mathcal{K}_0$. 
Let $\varphi_t:=t\varphi_k+(1-t)\varphi$. By (\ref{firstvar}), we have
\begin{align}
\left|S_1[\omega_{\varphi_k},\omega]-S_1[\omega_{\varphi},\omega]\right|&=\left|\int_0^1 dt \int_M(\varphi_k -\varphi)(\widehat{a_j}+\Delta_t a_{j-1}(\omega_t)-a_{j}(\omega_t))\frac{\omega_t^n}{n!}\right| \nonumber\\
&\leq \int_0^1 dt ||\varphi_k -\varphi||_{C^0}\int_M \left|\widehat{a_j}+\Delta_t a_{j-1}(\omega_t)-a_{j}(\omega_t)\right|\frac{\omega_t^n}{n!} \leq Ck^{-2}, \label{k22}
\end{align}
for large enough $k$, where we used (\ref{k2}) and the fact that $\omega_{\varphi_k}\rightarrow \omega_\varphi$ in $C^\infty$ to get constant $C$ uniform over bounded subsets in $\mathcal{K}_0$. Now observe that 
\begin{equation*}
(2\pi)^n\log \det\nolimits_{\omega} (Hilb_k(\varphi))-(2\pi)^n d_k\log d_k/V = (2\pi)^n\log \frac{Z_k[\varphi]}{Z_k[0]}.
\end{equation*}
By asymptotics (\ref{ver2}) and (\ref{k22}), we get
\begin{equation*}
\left| S_{L,k}(Hilb_k(\varphi))- S_2[\omega_\varphi,\omega] \right| \leq  k\left| S_1[\omega_\varphi,\omega] - S_1[\omega_{\varphi_k},\omega]\right| +C'k^{-1}\leq C k^{-1},
\end{equation*}
using the assumption $S_0[\varphi,0]=0$.
\end{proof}

\section{Derivation of the (2n+1)D Chern-Simons action}\label{sectioncs}

Recall that, in physics literature, $\Psi^k$ from the definition of a determinantal point process corresponds to the integer quantum Hall wave function; see \cite{tong2016lectures,klevtsov2016geometry}. Denote by $A_\mu$ the connection ($U(1)$-gauge field) on $L$. An effective action $S_{eff}[A_\mu]$ for the integer QHE is defined by
\begin{equation*}
Z[A_\mu]=e^{iS_{eff}[A_\mu]}.
\end{equation*}
The functional derivative of the effective action with respect to the time component of $A_\mu$ is given by
\begin{equation}\label{karabali}
\frac{\delta S_{eff}}{\delta A_0} = J_0,
\end{equation}
where $J_0$ is the charge density. In \cite[(26), (41)]{karabali2016geometry}, Karabali-Nair used (\ref{karabali}) to derive an effective action for the higher-dimensional QHE in terms of the Chern-Simons forms integrated over a $2n+1$ dimensional manifold. In this section, we present an alternative way of deriving their formula and show that as $k\rightarrow\infty$, the leading-order term is the $2n+1$ dimensional Chern-Simons action. The following construction is motivated by \cite[Proposition 1.4]{tian1999bott}.

Let $\varphi_1 \in \mathcal{K}_\omega$ and choose a smooth path $\varphi_t$ in $\mathcal{K}_\omega$ joining $0$ and $\varphi_1$. Let $\varphi : \mathbb{C} \rightarrow \mathcal{K}_\omega$ be a smooth map defined by $\varphi(z)=\varphi_t$ where $z=1-t+is\in\mathbb{C}$ and trivially extended over $t\notin [0,1]$. Define Hermitian metrics $\mathbf{h}$ and $\boldsymbol{\omega}$ on the pull-back bundles $\mathbf{L}:=\pi_1^*L$ and $\mathbf{T'M}:=\pi_1^*T^{1,0}M$, where $\pi_1 : M\times\mathbb{C}\rightarrow M$ is the projection, by
\begin{equation*}
\left.\mathbf{h}\right|_{M\times\{z\}}=\left.h_{\varphi(z)}\right|_M=\left.he^{-\varphi(z)}\right|_M,
\end{equation*}
and
\begin{equation*}
\left.\boldsymbol{\omega}\right|_{M\times\{z\}}=\left.\omega_{\varphi(z)}\right|_M.
\end{equation*}
Also, define Hermitian metrics $\mathbf{h_0}$ and $\boldsymbol{\omega_0}$ on $\mathbf{L}$ and $\mathbf{T'M}$ by
\begin{equation*}
\left.\mathbf{h_0}\right|_{M\times\{z\}}=\left.h\right|_M,
\end{equation*}
and
\begin{equation*}
\left.\boldsymbol{\omega_0}\right|_{M\times\{z\}}=\left.\omega\right|_M.
\end{equation*}
That is, $\mathbf{h_0}=\pi_1^*h$ and $\boldsymbol{\omega_0}=\pi_1^*\omega$.

\begin{lem}\label{tianlem}
On $\mathbb{C}$, we have
\begin{align*}
&\p_z \int_M BC(\Td;\omega_{\varphi(z)},\omega)\ch(R_L(h))+\Td(R_{T^{1,0}M}(\omega_{\varphi(z)}))BC(\ch;h_{\varphi(z)},h) \\
&=\int_M \left[ CS(\Td;\boldsymbol{\nabla}_\mathbf{T'M},\boldsymbol{\nabla^0}_\mathbf{T'M})\ch(R_\mathbf{L}(\mathbf{h_0}))+\Td(R_{\mathbf{T'M}}(\boldsymbol{\omega}))CS(\ch;\boldsymbol{\nabla}_\mathbf{L},\boldsymbol{\nabla^0}_\mathbf{L})\right]_{2n+1},
\end{align*}
where $\boldsymbol{\nabla}_\mathbf{L}$, $\boldsymbol{\nabla}_\mathbf{T'M}$ are Chern connections on $\mathbf{L}$, $\mathbf{T'M}$ associated with $\mathbf{h}$, $\boldsymbol{\omega}$, respectively, and $\boldsymbol{\nabla^0}_\mathbf{L}$, $\boldsymbol{\nabla^0}_\mathbf{T'M}$ are Chern connections associated with $\mathbf{h_0}$, $\boldsymbol{\omega_0}$, respectively.
\end{lem}
\begin{proof}
Let $f$ be an arbitrary smooth (0,1)-form with compact support in $\mathbb{C}$. Then we have
\begin{align*}
&\int_{\mathbb{C}} \p_z f \int_M BC(\Td;\omega_{\varphi(z)},\omega)\ch(R_L(h))+\Td(R_{T^{1,0}M}(\omega_{\varphi(z)}))BC(\ch;h_{\varphi(z)},h) \\
&=\int_{M\times\mathbb{C}} \left(\p f\right) \left[ BC(\Td;\boldsymbol{\omega},\boldsymbol{\omega_0})\ch(R_\mathbf{L}(\mathbf{h_0}))+\Td(R_{\mathbf{T'M}}(\boldsymbol{\omega}))BC(\ch;\mathbf{h},\mathbf{h_0}) \right]_{2n} \\
&=\int_{M\times\mathbb{C}} f\left[ \p BC(\Td;\boldsymbol{\omega},\boldsymbol{\omega_0})\ch(R_\mathbf{L}(\mathbf{h_0})) + \Td(R_{\mathbf{T'M}}(\boldsymbol{\omega}))\p BC(\ch;\mathbf{h},\mathbf{h_0}) \right]_{2n+1} \\
&= \int_{\mathbb{C}} f \int_M \left[ CS(\Td;\boldsymbol{\nabla}_\mathbf{T'M},\boldsymbol{\nabla^0}_\mathbf{T'M})\ch(R_\mathbf{L}(\mathbf{h_0}))+\Td(R_{\mathbf{T'M}}(\boldsymbol{\omega}))CS(\ch;\boldsymbol{\nabla}_\mathbf{L},\boldsymbol{\nabla^0}_\mathbf{L})\right]_{2n+1},
\end{align*}
where the first identity is obtained by the fact that $\p f$ is of the form $\tilde{f}(z)dz\wedge d\bz$, the second identity is obtained by integration by parts, and the last identity is obtained by property (\ref{bccs}).
\end{proof}

From (\ref{Qu}) and (\ref{anomaly}), we have an expression of the effective action in terms of the Bott-Chern forms (ignoring $2\pi$ factors, assuming $\log Z_1[0]=0$ and higher cohomology of $L$ vanishes):
\begin{align}\label{effbc}
&S_{eff}[A_\mu]= -i\log Z_1[\varphi_1]\nonumber \\
&= -\int_M BC(\Td;\omega_{\varphi_1},\omega)\ch(R_{L}(h))+\Td(R_{T^{1,0}M}(\omega_{\varphi_1}))BC(\ch;h_{\varphi_1},h) - 2i\log \frac{T(\omega_{\varphi_1},h_{\varphi_1})}{T(\omega,h)}.
\end{align}
Denote $\widetilde{S}:=-2i\log \frac{T(\omega_{\varphi_1},h_{\varphi_1})}{T(\omega,h)}$. Now we prove Proposition \ref{m5}.

\begin{proof}[Proof of Proposition \ref{m5}]
By Lemma \ref{tianlem}, we have
\begin{align*}
&-\int_M BC(\Td;\omega_{\varphi_1},\omega)\ch(R_{L}(h))+\Td(R_{T^{1,0}M}(\omega_{\varphi_1}))BC(\ch;h_{\varphi_1},h)\\
&=-\int_0^1 \frac{\p}{\p t}\left( \int_M BC(\Td;\omega_{\varphi_t},\omega)\ch(R_{L}(h))+\Td(R_{T^{1,0}M}(\omega_{\varphi_t}))BC(\ch;h_{\varphi_t},h)\right)\wedge dt\\
&=\int_0^1 2 \p_z \int_M BC(\Td;\omega_{\varphi(z)},\omega)\ch(R_{L}(h))+\Td(R_{T^{1,0}M}(\omega_{\varphi(z)}))BC(\ch;h_{\varphi(z)},h)\\
&=\int_0^1 2\int_M \left[ CS(\Td;\boldsymbol{\nabla}_\mathbf{T'M},\boldsymbol{\nabla^0}_\mathbf{T'M})\ch(R_\mathbf{L}(\mathbf{h_0}))+\Td(R_{\mathbf{T'M}}(\boldsymbol{\omega}))CS(\ch;\boldsymbol{\nabla}_\mathbf{L},\boldsymbol{\nabla^0}_\mathbf{L})\right]_{2n+1}.
\end{align*}
Substituting it into (\ref{effbc}), we get the formula (\ref{eff}). The last claim follows immediately.
\end{proof}

It would be interesting to understand the relation between (\ref{eff}) and \cite[Theorem 3]{klevtsov2017quantum} in dimension $n=1$.

\appendix \section{Explicit computations for the third coefficient}\label{appendix}

Recall the $1$-forms $\gamma^{(j)}$ defined on $\mathcal{K}_\omega$ by (\ref{1forms}). In this appendix, we show that $\gamma^{(2)}$ is closed and obtain the second variation formula for $S_2$. Note that here we do not assume the polarization of $[\omega]$. We start with some standard identities in K\"ahler geometry. A good reference for K\"ahler geometry is \cite{szekelyhidi2014introduction}. Let $\a, \b$ be $(1,1)$-forms given by $\a=i\a_{j\bk}dz^j\wedge d\bz^k, \b=i\b_{j\bk}dz^j\wedge d\bz^k$ such that $\a_{j\bk},\b_{j\bk}$ are Hermitian matrices. Then we have
\begin{align}
n\a\wedge\omega^{n-1}&=\left(\tr_{\omega}\a\right)\omega^n ;\label{computation1}\\
n(n-1)\a\wedge\b\wedge\omega^{n-2}&=\left[ \left(\tr_{\omega}\a\right) \left(\tr_{\omega}\b\right) -\langle\a,\b\rangle_{\omega} \right]\omega^n,\label{computation2}
\end{align}
where $\tr_{\omega}\a:=g^{j\bk}\a_{j\bk}$ and $\langle\a,\b\rangle_{\omega}:=g^{j\bk}g^{r\bl}\a_{j\bl}\b_{r\bk}$. Let $\omega_t:=\omega+it\p\bp\varphi\in\mathcal{K}_0$ for $t$ near $0$. We collect some variation formulas for the associated geometric quantities in the following. We denote by $\ric$ the Ricci form defined by $\ric:=i\ric_{j\bk}dz^j\wedge d\bz^k$.
\begin{align*}
\frac{d}{dt}\ric&=-i\p\bp\Delta\varphi ; \quad \frac{d}{dt} S=-\Delta^2 \varphi - \langle i\p\bp\varphi,\ric \rangle_{\omega};\\
\frac{d}{dt}\Delta S &= -\Delta^3\varphi -\Delta\langle i\p\bp\varphi,\ric \rangle_{\omega}-\langle i\p\bp\varphi,i\p\bp S\rangle_{\omega}.
\end{align*}
Finally, note that for functions $f, g$ and closed $(n-1,n-1)$-form $T$, the expression $\int_M f i\p\bp g\wedge T$ is symmetric in $f$ and $g$, by integration by parts.

\begin{thm}
Define $1$-form $\gamma^{(2)}$ on $\mathcal{K}_\omega$ by
\begin{equation*}
\gamma^{(2)}_{\varphi}(\psi):=\int_M \psi \left( \frac{1}{6}\Delta_\varphi S_\varphi-\frac{1}{24}\left(|R_\varphi|^2-4|\ric_\varphi|^2+3S_\varphi^2\right) \right) \frac{\omega_\varphi^n}{n!}
\end{equation*}
for $\psi\in T_{\varphi}\mathcal{K}_\omega=C^{\infty}(M,\mathbb{R})$. Then $\gamma^{(2)}$ is closed.
\end{thm}

\begin{proof}[Proof 1]
In this first proof, we prove $\gamma^{(2)}$ is exact. That is, let $\widetilde{S} : \mathcal{K}_\omega \rightarrow \mathbb{R}$ by 
\begin{equation}\label{exact}
\widetilde{S}(\varphi)=\int_M -i BC(\Td_2;\omega_\varphi,\omega)\frac{\omega^{n-1}}{(n-1)!}+\Td_{2}(R_\varphi)\frac{-1}{(n-1)!}\sum^{n-2}_{s=0}\varphi \omega_{\varphi}^s\wedge \omega^{n-2-s}.
\end{equation}
Let $\psi\in T_{\varphi}\mathcal{K}_\omega$. We will show that $d\widetilde{S}_\varphi(\psi)=\left.\frac{d}{dt} \right|_{t=0}\widetilde{S}(\varphi+t\psi)=\gamma^{(2)}_{\varphi}(\psi)$. It is clear that by virtue of cocyle identity (see Remark \ref{remarkcocycle}), we can assume $\varphi=0$ without loss of generality. First, we have 
\begin{equation*}
\Td_2 = \frac{1}{12}(\cc_1^2+\cc_2)=\frac{1}{24}(-\cc_1^2+2\cc_2+3\cc_1^2)=\frac{1}{24}(\Tr_2-3\Tr_1^2)
\end{equation*}
where we denote by $\Tr_j$ the $j^{th}$ trace polynomial $\Tr_j(A):=\Tr(A^j)$. We use property \textit{(3)} of Proposition \ref{bc} to compute :
\begin{align*}
\left.\frac{d}{dt} \right|_{t=0} BC(\Tr_2;\omega_{t\psi},\omega)\frac{\omega^{n-1}}{(n-1)!}&=2\Tr[R\dot{\omega}\omega^{-1}]\frac{\omega^{n-1}}{(n-1)!}\\
=2R_{p}{}^{q}{}_{j\bk}\p_{q}\p_{\bl}\psi g^{p\bl} dz^j\wedge d\bz^k\frac{\omega^{n-1}}{(n-1)!} &= \frac{2}{i}\ric_p^q \p_{q}\p_{\bl}\psi g^{p\bl} \frac{\omega^n}{n!} = \frac{2}{i} \langle \ric,i\p\bp\psi \rangle_\omega \frac{\omega^n}{n!},
\end{align*}
and
\begin{align*}
\left.\frac{d}{dt} \right|_{t=0} BC(\Tr_1^2;\omega_{t\psi},\omega)\frac{\omega^{n-1}}{(n-1)!}&=2\Tr[R]\Tr[\dot{\omega}\omega^{-1}]\frac{\omega^{n-1}}{(n-1)!}\\
=\frac{2}{i}\Delta\psi\ric\wedge \frac{\omega^{n-1}}{(n-1)!}&=\frac{2}{i}\Delta\psi S \frac{\omega^{n}}{n!},
\end{align*}
where we used (\ref{computation1}) and (\ref{computation2}). Since (\ref{computation1}) and (\ref{computation2}) will be used frequently from now on, we will not mention them each time they are used. In summary, we get
\begin{align}\label{ex1}
\left.\frac{d}{dt} \right|_{t=0} \int_M -iBC(\Td_2;\omega_{t\psi},\omega)\frac{\omega^{n-1}}{(n-1)!}=\int_M -\frac{1}{12}\langle \ric,i\p\bp\psi \rangle_\omega \frac{\omega^{n}}{n!} &+\frac{1}{4}\Delta\psi S \frac{\omega^{n}}{n!}\nonumber\\
=\int_M \frac{1}{12} \ric\wedge i\p\bp\psi\wedge\frac{\omega^{n-2}}{(n-2)!}-\frac{1}{12}S\Delta\psi\frac{\omega^{n}}{n!}+\frac{1}{4}\Delta\psi S \frac{\omega^{n}}{n!} &=\int_M \frac{1}{6}\psi \Delta S\frac{\omega^{n}}{n!}.
\end{align}
For the second term in (\ref{exact}), we have
\begin{align}\label{ex2}
\left.\frac{d}{dt} \right|_{t=0}& \int_M \Td_2(R)\frac{-1}{(n-1)!}\sum^{n-2}_{s=0}t\psi \omega_{t\psi}^s\wedge \omega^{n-2-s} = \int_M -\Td_2(R)\psi\frac{\omega^{n-2}}{(n-2)!} \nonumber \\
&= \int_M \frac{1}{24}\left( -\Tr_2(R)+3\Tr_1^2(R) \right)\psi \frac{\omega^{n-2}}{(n-2)!}\nonumber \\
&= \int_M \frac{1}{24}\psi\left(|\ric|^2-|R|^2\right)\frac{\omega^{n-2}}{(n-2)!}-\frac{3}{24}\psi\ric\wedge\ric\wedge\frac{\omega^{n-2}}{(n-2)!}\nonumber\\
&=\int_M -\frac{1}{24}\psi\left(|R|^2-4|\ric|^2+3S^2\right)\frac{\omega^{n-2}}{(n-2)!}.
\end{align}
Combining (\ref{ex1}) and (\ref{ex2}), we prove the claim.
\end{proof}

\begin{proof}[Proof 2]
In this second proof, we directly prove $\gamma^{(2)}$ is closed in the spirit of \cite{mabuchi1986k}. Let $\psi_1, \psi_2 \in T_{\varphi}\mathcal{K}_\omega=C^{\infty}(M,\mathbb{R})$. We will show that $d\gamma^{(2)}_\varphi(\psi_1,\psi_2)=0$. Note that since
\begin{equation*}
d\gamma^{(2)}_\varphi(\psi_1,\psi_2)=\psi_1\cdot\gamma^{(2)}(\psi_2)-\psi_2\cdot\gamma^{(2)}(\psi_1),
\end{equation*}
we only need to show $\psi_1\cdot\gamma^{(2)}(\psi_2)=\left.\frac{d}{dt}\right|_{t=0}\gamma^{(2)}_{\varphi+t\psi_1}(\psi_2)$ is symmetric in $\psi_1$ and $\psi_2$. Assume $\varphi=0$ without loss of generality. First we divide $\gamma^{(2)}$ into three parts :
\begin{align*}
24n!\gamma^{(2)}_0(\psi) &= \int_M 4\psi\Delta S \omega^n + \int_M 3\psi\left(|\ric|^2-S^2\right)\omega^n +\int_M \psi \left(|\ric|^2-|R|^2\right)\omega^n\\
&=:I_0(\psi)+II_0(\psi)+III_0(\psi).
\end{align*}

\textit{i) $\psi_1\cdot I(\psi_2)$}
\begin{align}\label{final1}
&\left.\frac{d}{dt}\right|_{t=0}I_{t\psi_1}(\psi_2)= 4\int_M \psi_2 \left.\frac{d}{dt}\right|_{t=0}\Delta_{t\psi_1} S_{\psi_1} \omega_{t\psi_1}^n  \nonumber\\
&\;= 4\int_M \psi_2 \left( -\Delta^3\psi_1 -\Delta\langle i\p\bp\psi_1,\ric \rangle-\langle i\p\bp\psi_1,i\p\bp S\rangle + \Delta S \Delta \psi_1\right)\omega^n \nonumber\\
&\;=4\int_M -\psi_2\Delta^3\psi_1 \omega^n -\Delta\psi_2\langle i\p\bp\psi_1,\ric \rangle\omega^n + n(n-1)\psi_2 i\p\bp\psi_1\wedge i\p\bp S\wedge\omega^{n-2}.
\end{align}
Note that $\int_M -\psi_2\Delta^3\psi_1 \omega^n$ and $\int_M \psi_2 i\p\bp\psi_1\wedge i\p\bp S\wedge\omega^{n-2}$ are symmetric in $\psi_1$ and $\psi_2$. 

\vspace{5mm}
\textit{ii) $\psi_1\cdot II(\psi_2)$}
\begin{align}\label{final2}
&\left.\frac{d}{dt}\right|_{t=0}II_{t\psi_1}(\psi_2)= -3\left.\frac{d}{dt}\right|_{t=0}\int_M \psi_2 \left(S_{t\psi_1}^2-|\ric_{t\psi_1}|^2\right)\omega_{t\psi_1}^n \nonumber\\
&\;= -3\left.\frac{d}{dt}\right|_{t=0}\int_M n(n-1)\psi_2\ric_{r\psi_1}\wedge\ric_{r\psi_1}\wedge\omega_{t\psi_1}^{n-2} \nonumber\\
&\;=-3n(n-1)\int_M -2\psi_2i\p\bp\Delta\psi_1\wedge\ric\wedge\omega^{n-2}+(n-2)\psi_2\ric^2\wedge i\p\bp\psi_1\wedge\omega^{n-3} \nonumber\\
&\;=6n(n-1)\int_M \Delta\psi_1 i\p\bp\psi_2\wedge\ric\wedge\omega^{n-2}-3n(n-1)(n-2)\int_M \psi_2i\p\bp\psi_1\wedge\ric^2\wedge\omega^{n-3} \nonumber\\
&\;= 6 \int_M \Delta\psi_1 \Delta\psi_2 S \omega^n -6\int_M\Delta\psi_1\langle i\p\bp\psi_2,\ric \rangle\omega^n \nonumber\\
&\qquad\qquad\qquad\qquad\qquad\qquad\qquad\qquad-3n(n-1)(n-2)\int_M \psi_2i\p\bp\psi_1\wedge\ric^2\wedge\omega^{n-3}.
\end{align}
Note that $\int_M \Delta\psi_1 \Delta\psi_2 S \omega^n$ and $\int_M \psi_2i\p\bp\psi_1\wedge\ric^2\wedge\omega^{n-3}$ are symmetric in $\psi_1$ and $\psi_2$. 

\vspace{5mm}
\textit{iii) $\psi_1\cdot III(\psi_2)$}

\begin{align}\label{tr2cal}
&\left.\frac{d}{dt}\right|_{t=0}III_{t\psi_1}(\psi_2)= \left.\frac{d}{dt}\right|_{t=0}\int_M \psi_2\left(|\ric_{t\psi_1}|^2-|R_{t\psi_1}|^2\right)\omega_{t\psi_1}^n \nonumber\\
&\;=\left.\frac{d}{dt}\right|_{t=0} -n(n-1)\int_M \psi_2\Tr_2(R_{t\psi_1})\wedge\omega_{t\psi_1}^{n-2} \nonumber\\
&\;= -n(n-1)\int_M \psi_2 \left( \left.\frac{d}{dt}\right|_{t=0} \Tr_2(R_{t\psi_1}) \right) \wedge\omega^{n-2} \nonumber\\
&\qquad\qquad\qquad\qquad\qquad\qquad\qquad-n(n-1)(n-2)\int_M \psi_2 i\p\bp\psi_1\wedge\Tr(R^2)\wedge\omega^{n-3}. 
\end{align}
Note that $\int_M \psi_2 i\p\bp\psi_1\wedge\Tr(R^2)\wedge\omega^{n-3}$ is symmetric in $\psi_1$ and $\psi_2$. Now we compute the first term in (\ref{tr2cal}). By properties \textit{(2)} and \textit{(3)} of Proposition \ref{bc}, we have
\begin{equation*}
\left.\frac{d}{dt}\right|_{t=0} \Tr_2(R_{t\psi_1})=\bp\p 2\Tr[R\dot{\omega}\omega^{-1}]=2\bp\p \left(R_{p}{}^{q}{}_{j\bk}\p_{q}\p_{\bl}\psi_1 g^{p\bl} dz^j\wedge d\bz^k \right).
\end{equation*} 
Using this, we can compute
\begin{align}\label{final3}
& -n(n-1)\int_M \psi_2 \left( \left.\frac{d}{dt}\right|_{t=0} \Tr_2(R_{t\psi_1}) \right)\wedge \omega^{n-2}  \nonumber\\
&\;= -2n(n-1)\int_M \psi_2\bp\p \left(R_{p}{}^{q}{}_{j\bk}\p_{q}\p_{\bl}\psi_1 g^{p\bl} dz^j\wedge d\bz^k \right) \wedge\omega^{n-2}  \nonumber\\
&\;=2n(n-1)\int_M \p\bp\psi_2\wedge \left(R_{p\br j\bk}\p_{q}\p_{\bl}\psi_1 g^{p\bl}g^{q\br}dz^j\wedge d\bz^k \right)\wedge \omega^{n-2} \nonumber\\
&\;= \frac{2n(n-1)}{i}\int_M i\p\bp\psi_2 \wedge\langle iR_{j\bk},i\p\bp\psi_1 \rangle dz^j\wedge d\bz^k \wedge \omega^{n-2} \nonumber\\
&\;=-2n(n-1)\int_M i\p\bp\psi_2 \wedge i\langle iR_{j\bk},i\p\bp\psi_1 \rangle dz^j\wedge d\bz^k \wedge \omega^{n-2}  \nonumber\\
&\;=-2\int_M \left[ \Delta\psi_2 \langle \ric, i\p\bp\psi_1 \rangle - g^{i\bar{q}}g^{p\bar{j}}g^{r\bar{l}}g^{k\bar{s}} R_{p\bar{q}r\bar{s}} \p_i\p_{\bar{j}}\psi_2 \p_k\p_{\bar{l}}\psi_1 \right]\omega^n, 
\end{align}
where we used the fact $R_{p\bar{q} r\bl}=R_{r\bl p\bar{q}}$ to get the third identity. Note that $g^{i\bar{q}}g^{p\bar{j}}g^{r\bar{l}}g^{k\bar{s}} R_{p\bar{q}r\bar{s}} \p_i\p_{\bar{j}}\psi_2 \p_k\p_{\bar{l}}\psi_1$ is symmetric in $\psi_1$ and $\psi_2$. Combining (\ref{final1}), (\ref{final2}), and (\ref{final3}), we obtain

\begin{align*}
&\psi_1\cdot\gamma^{(2)}(\psi_2)=\left(\text{terms symmetric in $\psi_1$ and $\psi_2$} \right) \nonumber\\
&\quad +\frac{1}{24}\int_M \left[ -4\Delta\psi_2\langle i\p\bp\psi_1,\ric \rangle -6\Delta\psi_1\langle i\p\bp\psi_2,\ric \rangle -2\Delta\psi_2\langle i\p\bp\psi_1,\ric \rangle \right] \frac{\omega^n}{n!},
\end{align*}
which is symmetric in $\psi_1$ and $\psi_2$.
\end{proof}

\textit{Proof 1} shows that for any smooth path $\varphi_t$ in $\mathcal{K}_\omega$ joining $0$ and $\varphi$, we have
\begin{align*}
&\int_M -i BC(\Td_2;\omega_\varphi,\omega)\frac{\omega^{n-1}}{(n-1)!}+\Td_{2}(R_\varphi)\frac{-1}{(n-1)!}\sum^{n-2}_{s=0}\varphi \omega_{\varphi}^s\wedge \omega^{n-2-s} \\
&= \int_0^1 dt \int_M \dot{\varphi_t}\left( \frac{1}{6}\Delta_t S_t-\frac{1}{24}\left(|R_t|^2-4|\ric_t|^2+3S_t^2\right) \right)\frac{\omega_t^n}{n!},
\end{align*}
without polarization assumption on $[\omega]$.

From \textit{Proof 2}, we can read off the second variation formula of $S_2$, Proposition \ref{m2}.

\bibliographystyle{alpha}
\bibliography{References}

@article{berman2010growth,
  title={Growth of balls of holomorphic sections and energy at equilibrium},
  author={Berman, Robert J. and Boucksom, S{\'e}bastien},
  journal={Invent. Math.},
  volume={181},
  number={2},
  pages={337--394},
  year={2010},
  publisher={Springer}
}

@article{finski2025small,
  title={Small eigenvalues of {T}oeplitz operators, {L}ebesgue envelopes and {M}abuchi geometry},
  author={Finski, Siarhei},
  journal={arXiv preprint arXiv:2502.01554},
  year={2025}
}

@article{berman2014determinantal,
  title={Determinantal point processes and fermions on complex manifolds: large deviations and bosonization},
  author={Berman, Robert J.},
  journal={Commun. Math. Phys.},
  volume={327},
  pages={1--47},
  year={2014},
  publisher={Springer}
}

@incollection{berman2013determinantal,
  title={Determinantal point processes and fermions on polarized complex manifolds: bulk universality},
  author={Berman, Robert J.},
  booktitle={Algebraic and analytic microlocal analysis},
  pages={341--393},
  year={2013},
  publisher={Springer}
}

@incollection{berman2018kahler,
  title={Kahler-{E}instein metrics, canonical random point processes and birational geometry},
  author={Berman, Robert J.},
  booktitle={Algebraic geometry: Salt Lake City 2015},
  series={Proceedings of Symposia in Pure Mathematics},
  pages={29--73},
  year={2018},
  publisher={American Mathematical Society, Providence}
}

@article{berman2017large,
  title={Large deviations for {G}ibbs measures with singular {H}amiltonians and emergence of {K}{\"a}hler--{E}instein metrics},
  author={Berman, Robert J.},
  journal={Commun. Math. Phys.},
  volume={354},
  pages={1133--1172},
  year={2017},
  publisher={Springer}
}

@article{berman2022measure,
  title={Measure preserving holomorphic vector fields, invariant anti-canonical divisors and {G}ibbs stability},
  author={Berman, Robert J.},
  journal={Anal. Math.},
  volume={48},
  number={2},
  pages={347--375},
  year={2022},
  publisher={Springer}
}

@article{berman2024probabilistic,
  title={The probabilistic vs the quantization approach to {K}{\"a}hler--{E}instein geometry},
  author={Berman, Robert J.},
  journal={Math. Ann.},
  volume={388},
  number={4},
  pages={4383--4404},
  year={2024},
  publisher={Springer}
}

@incollection{berman2011bergman,
  title={Bergman geodesics},
  author={Berman, Robert J. and Keller, Julien},
  booktitle={Complex Monge--Amp{\`e}re Equations and Geodesics in the Space of K{\"a}hler Metrics},
  pages={283--302},
  year={2011},
  publisher={Springer}
}

@article{aoi2024microscopic,
  title={Microscopic stability thresholds and constant scalar curvature  {K}{\"a}hler metrics},
  author={Aoi, Takahiro},
  journal={arXiv preprint arXiv:2410.22090},
  year={2024}
}

@article{fujita2016berman,
  title={On {B}erman--{G}ibbs stability and {K}-stability of {Q}-{F}ano varieties},
  author={Fujita, Kento},
  journal={Compos. Math.},
  volume={152},
  number={2},
  pages={288--298},
  year={2016},
  publisher={London Mathematical Society}
}

@article{fujita2018k,
  title={On the {K}-stability of {F}ano varieties and anticanonical divisors},
  author={Fujita, Kento and Odaka, Yuji},
  journal={Tohoku Math. J., Second Series},
  volume={70},
  number={4},
  pages={511--521},
  year={2018},
  publisher={Mathematical Institute, Tohoku University}
}

@article{tian1990set,
  title={On a set of polarized {K}{\"a}hler metrics on algebraic manifolds},
  author={Tian, Gang},
  journal={J. Differ. Geom.},
  volume={32},
  number={1},
  pages={99--130},
  year={1990},
  publisher={Lehigh University}
}

@inproceedings{catlin1999bergman,
  title={The {B}ergman kernel and a theorem of {T}ian},
  author={Catlin, David},
  booktitle={Analysis and Geometry in Several Complex Variables: Proceedings of the 40th Taniguchi Symposium},
  pages={1--23},
  year={1999},
  organization={Springer}
}

@article{zelditch1998szego,
  title={Szeg{\"o} kernels and a theorem of {T}ian},
  author={Zelditch, Steve},
  journal={Int. Math. Res. Not.},
  volume={1998},
  number={6},
  pages={317--331},
  year={1998},
  publisher={OUP}
}

@article{berman2008direct,
  title={A direct approach to {B}ergman kernel asymptotics for positive line bundles},
  author={Berman, Robert J. and Berndtsson, Bo and Sj{\"o}strand, Johannes},
  journal={Ark. Mat.},
  volume={46},
  pages={197--217},
  year={2008},
  publisher={Springer}
}

@article{douglas2010bergman,
  title={Bergman kernel from path integral},
  author={Douglas, Michael R and Klevtsov, Semyon},
  journal={Commun. Math. Phys.},
  volume={293},
  pages={205--230},
  year={2010},
  publisher={Springer}
}

@article{xu2012closed,
  title={A closed formula for the asymptotic expansion of the {B}ergman kernel},
  author={Xu, Hao},
  journal={Commun. Math. Phys.},
  volume={314},
  pages={555--585},
  year={2012},
  publisher={Springer}
}

@article{hezari2016asymptotic,
  title={Asymptotic expansion of the {B}ergman kernel via perturbation of the {B}argmann--{F}ock model},
  author={Hezari, Hamid and Kelleher, Casey and Seto, Shoo and Xu, Hang},
  journal={J. Geom. Anal.},
  volume={26},
  pages={2602--2638},
  year={2016},
  publisher={Springer}
}

@article{bismut1988analytic,
  title={Analytic torsion and holomorphic determinant bundles, {I-III}},
  author={Bismut, Jean-Michel and Gillet, Henri and Soul{\'e}, Christophe},
  journal={Commun. Math. Phys.},
  volume={115},
  pages={49--129, 301--351},
  year={1988},
  publisher={Springer}
}

@article{vasserot1989asymptotics,
  title={The asymptotics of the {R}ay-{S}inger analytic torsion associated with high powers of a positive line bundle},
  author={Bismut, Jean-Michel and Vasserot, {\'E}ric},
  journal={Commun. Math. Phys.}, 
  volume={125},
  number={2},
  pages={355--367},
  year={1989},
  publisher={Springer}
}

@article{finski2018full,
  title={On the full asymptotics of analytic torsion},
  author={Finski, Siarhei},
  journal={J. Funct. Anal.},
  volume={275},
  number={12},
  pages={3457--3503},
  year={2018},
  publisher={Elsevier}
}

@article{donaldson1985anti,
  title={Anti self-dual {Y}ang-{M}ills connections over complex algebraic surfaces and stable vector bundles},
  author={Donaldson, Simon K},
  journal={Proc. Lond. Math. Soc.},
  volume={3},
  number={1},
  pages={1--26},
  year={1985},
  publisher={Wiley Online Library}
}

@article{donaldson1987infinite,
  title={Infinite determinants, stable bundles and curvature},
  author={Donaldson, Simon K},
  journal={Duke Math. J.},
  volume={54},
  number={1},
  pages={231--247},
  year={1987},
  publisher={Duke University Press}
}

@article{donaldson2001scalar,
  title={Scalar curvature and projective embeddings, {I}},
  author={Donaldson, Simon K},
  journal={J. Differ. Geom.},
  volume={59},
  number={3},
  pages={479--522},
  year={2001},
  publisher={Lehigh University}
}

@article{donaldson2005scalar,
  title={Scalar curvature and projective embeddings, {II}},
  author={Donaldson, Simon K},
  journal={Q. J. Math.},
  volume={56},
  number={3},
  pages={345--356},
  year={2005},
  publisher={Oxford University Press}
}

@article{dai2006asymptotic,
  title={On the asymptotic expansion of {B}ergman kernel},
  author={Dai, Xianzhe and Liu, Kefeng and Ma, Xiaonan},
  journal={J. Differ. Geom.},
  volume={72},
  number={1},
  pages={1--41},
  year={2006},
  publisher={Lehigh University}
}

@book{ma2007holomorphic,
  title={Holomorphic Morse inequalities and Bergman kernels},
  author={Ma, Xiaonan and Marinescu, George},
  volume={254},
  year={2007},
  publisher={Springer Science \& Business Media}
}

@article{klevtsov2017quantum,
  title={Quantum {H}all effect and {Q}uillen metric},
  author={Klevtsov, Semyon and Ma, Xiaonan and Marinescu, George and Wiegmann, Paul},
  journal={Commun. Math. Phys.},
  volume={349},
  pages={819--855},
  year={2017},
  publisher={Springer}
}

@article{klevtsov2014random,
  title={Random normal matrices, {B}ergman kernel and projective embeddings},
  author={Klevtsov, Semyon},
  journal={J. High Energy Phys.},
  volume={2014},
  number={1},
  pages={1--19},
  year={2014},
  publisher={Springer}
}

@article{ferrari2013random,
  title={Random {K}{\"a}hler metrics},
  author={Ferrari, Frank and Klevtsov, Semyon and Zelditch, Steve},
  journal={Nucl. Phys. B},
  volume={869},
  number={1},
  pages={89--110},
  year={2013},
  publisher={Elsevier}
}

@article{ferrari2012gravitational,
  title={Gravitational actions in two dimensions and the {M}abuchi functional},
  author={Ferrari, Frank and Klevtsov, Semyon and Zelditch, Steve},
  journal={Nucl. Phys. B},
  volume={859},
  number={3},
  pages={341--369},
  year={2012},
  publisher={Elsevier}
}

@article{bilal20142d,
  title={2{D} quantum gravity at one loop with {L}iouville and {M}abuchi actions},
  author={Bilal, Adel and Ferrari, Frank and Klevtsov, Semyon},
  journal={Nucl. Phys. B},
  volume={880},
  pages={203--224},
  year={2014},
  publisher={Elsevier}
}

@article{weinkove2002higher,
  title={Higher {K}-energy functionals and higher {F}utaki invariants},
  author={Weinkove, Ben},
  journal={arXiv preprint math/0204271},
  year={2002}
}

@article{BC,
  title={Hermitian vector bundles and the equidistribution of the zeroes of their holomorphic sections},
  author={Bott, Raoul and Chern, Shiing-Shen},
  journal={Acta Math.},
  volume={114},
  pages={71--112},
  year={1965}
}

@article{pingali2014bott,
  title={On {B}ott--{C}hern forms and their applications},
  author={Pingali, Vamsi P and Takhtajan, Leon A},
  journal={Math. Ann.},
  volume={360},
  pages={519--546},
  year={2014},
  publisher={Springer}
}

@article{lu2000lower,
  title={On the lower order terms of the asymptotic expansion of {T}ian-{Y}au-{Z}elditch},
  author={Lu, Zhiqin},
  journal={Am. J. Math.},
  volume={122},
  number={2},
  pages={235--273},
  year={2000},
  publisher={Johns Hopkins University Press}
}

@article{lu2004k,
  title={K Energy and {K} Stability on Hypersurfaces},
  author={Lu, Zhiqin},
  journal={Commun. Anal. Geom.},
  volume={12},
  number={3},
  pages={601--630},
  year={2004},
  publisher={International Press}
}

@article{lu2004log,
  title={The log term of the {S}zeg{\"o} kernel},
  author={Lu, Zhiqin and Tian, Gang},
  journal={Duke Math. J.},
  volume={125},
  number={2},
  pages={351--387},
  year={2004}
}

@incollection{tianresidue,
  title={{K}{\"a}hler-{E}instein metrics on algebraic manifolds},
  author={Tian, Gang},
  booktitle={ Transcendental methods in algebraic
geometry (Cetraro, 1994)},
  pages={143-–185},
  volume={1646},
  year={1996},
  series={Lecture Notes in Mathematics},
  publisher={Springer}
}

@article{tian1999bott,
  title={Bott-{C}hern forms and geometric stability},
  author={Tian, Gang},
  journal={Discrete Contin. Dyn. Syst.},
  volume={6},
  number={1},
  pages={211--220},
  year={2000},
  publisher={Discrete Contin. Dyn. Syst.}
}

@book{tian2012canonical,
  title={Canonical metrics in {K}{\"a}hler geometry},
  author={Tian, Gang},
  year={2000},
  series={Lectures in Mathematics ETH Zürich},
  publisher={Birkh{\"a}user}
}

@article{tian1994k,
  title={The {K}-energy on hypersurfaces and stability},
  author={Tian, Gang},
  journal={Commun. Anal. Geom.},
  volume={2},
  number={2},
  pages={239--265},
  year={1994},
  publisher={International Press of Boston}
}

@article{chen2002ricci,
  title={Ricci flow on {K}{\"a}hler-{E}instein surfaces},
  author={Chen, Xiu-Xiong and Tian, Gang},
  journal={Invent. Math.},
  volume={147},
  number={3},
  pages={487--544},
  year={2002},
  publisher={Springer}
}

@article{ruan1996canonical,
  title={Canonical coordinates and {B}ergman metrics},
  author={Ruan, Wei-Dong},
  journal={Commun. Anal. Geom.},
  volume={6},
  number={3},
  pages={589--631},
  year={1998},
  publisher={International Press}
}

@article{bando2006obstruction,
  title={An obstruction for {C}hern class forms to be harmonic},
  author={Bando, Shigetoshi},
  journal={Kodai Math. J.},
  volume={29},
  number={3},
  pages={337--345},
  year={2006},
  publisher={Department of Mathematics, Tokyo Institute of Technology}
}

@article{bando1986some,
  title={On some integral invariants on complex manifolds, {I}},
  author={Bando, Shigetoshi and Mabuchi, Toshiki},
  journal={Proc. Jpn. Acad.},
  volume={62},
  pages={197--200},
  year={1986}
}

@article{mabuchi1986k,
  title={K-energy maps integrating {F}utaki invariants},
  author={Mabuchi, Toshiki},
  journal={Tohoku Math. J., Second Series},
  volume={38},
  number={4},
  pages={575--593},
  year={1986},
  publisher={Mathematical Institute, Tohoku University}
}

@article{futaki1983compact,
  title={On compact {K}{\"a}hler manifolds of constant scalar curvatures},
  author={Futaki, Akito},
  journal={Proc. Jpn. Acad.},
  volume={59},
  number={8},
  pages={401--402},
  year={1983},
}

@book{futaki2006kahler,
  title={K{\"a}hler-{E}instein metrics and integral invariants},
  author={Futaki, Akito},
  volume={1314},
  year={1988},
  series={Lecture Notes in Mathematics},
  publisher={Springer}
}

@article{futaki2004asymptotic,
  title={Asymptotic {C}how semi-stability and integral invariants},
  author={Futaki, Akito},
  journal={Int. J. Math.},
  volume={15},
  number={09},
  pages={967--979},
  year={2004},
  publisher={World Scientific}
}

@article{phong2007lectures,
  title={Lectures on stability and constant scalar curvature},
  author={Phong, Duong Hong and Sturm, Jacob},
  journal={Curr. Dev. Math.},  
  pages={101--176},
  year={2007},
  publisher={International Press}
}

@book{szekelyhidi2014introduction,
  title={An Introduction to Extremal Kahler Metrics},
  author={Sz{\'e}kelyhidi, G{\'a}bor},
  volume={152},
  year={2014},
  publisher={American Mathematical Society}
}

@article{karabali2016geometry,
  title={Geometry of the quantum {H}all effect: {A}n effective action for all dimensions},
  author={Karabali, Dimitra and Nair, VP},
  journal={Phys. Rev. D},
  volume={94},
  number={2},
  pages={024022},
  year={2016},
  publisher={APS}
}

@article{karabali2023transport,
  title={Transport coefficients for higher dimensional quantum {H}all effect},
  author={Karabali, Dimitra and Nair, VP},
  journal={Phys. Rev. B},
  volume={108},
  number={20},
  pages={205155},
  year={2023},
  publisher={APS}
}

@article{agarwal2025fractional,
  title={Fractional quantum {H}all effect in higher dimensions},
  author={Agarwal, Abhishek and Karabali, Dimitra and Nair, VP},
  journal={Phys. Rev. D},
  volume={111},
  number={2},
  pages={025002},
  year={2025},
  publisher={APS}
}

@article{tong2016lectures,
  title={Lectures on the quantum {H}all effect},
  author={Tong, David},
  journal={arXiv preprint arXiv:1606.06687},
  year={2016}
}

@article{klevtsov2016geometry,
    author = "Klevtsov, Semyon",
    title = "{Geometry and large {N} limits in {L}aughlin states}",
    journal = "Trav. Math.",
    volume = "24",
    pages = "63--127",
    year = "2016"
}

@article{shen2025geometric,
  title={Geometric {Z}abrodin-{W}iegmann conjecture for integer {Q}uantum {H}all states},
  author={Shen, Shu and Yu, Jianqing},
  journal={Commun. Math. Phys.},
  volume={406},
  number={298},
  year={2025},
  publisher={Springer}
}

@article{zabrodin2006large,
  title={Large-{N} expansion for the 2D {D}yson gas},
  author={Zabrodin, A and Wiegmann, P},
  journal={J. Phys. A, Math. Gen.},
  volume={39},
  number={28},
  pages={8933},
  year={2006},
  publisher={IOP Publishing}
}

@article{byun2023partition,
  title={Partition functions of determinantal and {P}faffian {C}oulomb gases with radially symmetric potentials},
  author={Byun, Sung-Soo and Kang, Nam-Gyu and Seo, Seong-Mi},
  journal={Commun. Math. Phys.},
  volume={401},
  number={2},
  pages={1627--1663},
  year={2023},
  publisher={Springer}
}

@article{byun2025free,
author = {Byun, Sung-Soo and Kang, Nam-Gyu and Seo, Seong-Mi and Yang, Meng},
title = {Free energy of spherical {C}oulomb gases with point charges},
journal = {J. Lond. Math. Soc.},
volume = {112},
number = {3},
pages = {e70294},
doi = {https://doi.org/10.1112/jlms.70294},
url = {https://londmathsoc.onlinelibrary.wiley.com/doi/abs/10.1112/jlms.70294},
eprint = {https://londmathsoc.onlinelibrary.wiley.com/doi/pdf/10.1112/jlms.70294},
year = {2025}
}

@article{serfaty2024lectures,
  title={Lectures on {C}oulomb and {R}iesz gases},
  author={Serfaty, Sylvia},
  journal={arXiv preprint arXiv:2407.21194},
  year={2024}
}

@article{can2015geometry,
  title={Geometry of quantum {H}all states: {G}ravitational anomaly and transport coefficients},
  author={Can, Tankut and Laskin, Michael and Wiegmann, Paul B},
  journal={Ann. Phys.},
  volume={362},
  pages={752--794},
  year={2015},
  publisher={Elsevier}
}

@article{klevtsov2014stability,
  title={Stability and integration over {B}ergman metrics},
  author={Klevtsov, Semyon and Zelditch, Steve},
  journal={J. High Energy Phys.},
  volume={2014},
  number={7},
  pages={1--27},
  year={2014},
  publisher={Springer}
}

@article{cristofori2025third,
  title={On the third coefficient in the {TYCZ}--expansion of the epsilon function of {K\"a}hler--{E}instein manifolds},
  author={Cristofori, Simone and Zedda, Michela},
  journal={J. Geom. Phys.},
  volume={209},
  pages={105384},
  year={2025},
  publisher={Elsevier}
}

@article{foth2007manifold,
  title={The manifold of compatible almost complex structures and geometric quantization},
  author={Foth, T and Uribe, A},
  journal={Commun. Math. Phys.},
  volume={274},
  number={2},
  pages={357--379},
  year={2007},
  publisher={Springer}
}

@article{eum2025asymptotic,
  title={Asymptotic expansion of the variation of the {Q}uillen metric and its moment map interpretation},
  author={Eum, Kiyoon},
  journal={arXiv preprint arXiv:2510.25456},
  year={2025}
}

\end{document}